\documentclass[10pt]{amsart}
\usepackage{amssymb,amsmath,amsthm}
\usepackage{mathrsfs}
\usepackage{hyperref}
\usepackage{mathrsfs}
\usepackage{color}

\newcommand\blfootnote[1]{%
  \begingroup
  \renewcommand\thefootnote{}\footnote{#1}%
  \addtocounter{footnote}{-1}%
  \endgroup
}
\newtheorem{theorem}{Theorem}
\newtheorem{prop}{Proposition}
\newtheorem{lemma}{Lemma}
\newtheorem{remark}{Remark}
\newtheorem{claim}{Claim}

\newtheorem{cor}{Corollary}

\numberwithin{equation}{section}
\numberwithin{theorem}{section}
\numberwithin{definition}{section}
\numberwithin{cor}{section}
\numberwithin{prop}{section}
\numberwithin{remark}{section}
\numberwithin{claim}{section}
\numberwithin{lemma}{section}

\def\Xint#1{\mathchoice
  {\XXint\displaystyle\textstyle{#1}}%
  {\XXint\textstyle\scriptstyle{#1}}%
  {\XXint\scriptstyle\scriptscriptstyle{#1}}%
  {\XXint\scriptscriptstyle\scriptscriptstyle{#1}}%
  \!\int}
\def\XXint#1#2#3{{\setbox0=\hbox{$#1{#2#3}{\int}$}
  \vcenter{\hbox{$#2#3$}}\kern-.5\wd0}}

\def\dashint{\Xint-}

\author{Gang Liu, G\'abor Sz\'ekelyhidi}
\address{School of Mathematical Sciences, Shanghai Key Laboratory of PMMP, East China Normal University}
\email{gangliunw@gmail.com}
\address{Department of Mathematics, University of Notre Dame, Notre Dame, IN 46556}
\email{gszekely@nd.edu}
\title[Gromov-Hausdorff limits]{Gromov-Hausdorff limits of K\"ahler manifolds with Ricci curvature bounded below}
\date{}
\begin{document}
\begin{abstract}
  We show that non-collapsed Gromov-Hausdorff limits of polarized K\"ahler manifolds, with Ricci curvature bounded below, are normal projective varieties, and the metric singularities
of the limit space are precisely given by a countable union of analytic subvarieties.
This extends a fundamental result of Donaldson-Sun, in which 2-sided Ricci curvature bounds were assumed. As a basic ingredient we show that, under lower Ricci curvature bounds, almost Euclidean balls in K\"ahler manifolds admit good holomorphic coordinates. Further applications are integral bounds for the scalar curvature on balls, and a rigidity theorem for K\"ahler manifolds with almost Euclidean volume growth.
\end{abstract}
\maketitle

\section{Introduction}
\blfootnote{The first author is partially supported by the Science and Technology Commission of Shanghai Municipality No. 18dz2271000. The second author is supported in part by NSF grant DMS-1350696}

The structure of Gromov-Hausdorff limits of Riemannian manifolds with Ricci curvature bounded below has been studied extensively since the seminal work of  Cheeger-Colding~\cite{CC,[CC2],[CC3],[CC4]}, with a great deal of more recent important progress (see e.g. \cite{CN}\cite{CN1}\cite{CN2}\cite{CJN}). In the K\"ahler setting, the recent breakthrough work of Donaldson-Sun~\cite{DS1} has led to many important advances. They proved in particular that the Gromov-Hausdorff limit of a sequence of non-collapsed, polarized K\"ahler manifolds, with 2-sided Ricci curvature bounds, is a normal projective variety.
Our first result is a generalization of this statement, removing the assumption of an upper bound for the Ricci curvature.

\begin{theorem}\label{thm:main}
  Given $n,d, v > 0$, there are constants $k_1, N > 0$ with the following property. Let $(M^n_i, L_i, \omega_i)$ be a sequence of polarized K\"ahler manifolds such that
  \begin{itemize}
    \item $L_i$ is a Hermitian holomorphic line bundle with curvature $-\sqrt{-1}\omega_i$;
    \item $\mathrm{Ric}(\omega_i) > - \omega_i$, $\mathrm{vol}(M_i) > v$, and $\mathrm{diam}(M_i, \omega_i) < d$;
    \item The sequence $(M^n_i, \omega_i)$ converges in the Gromov-Hausdorff sense to a limit metric space $X$.
    \end{itemize}
    Then each $M_i^n$ can be embedded in a subspace of $\mathbb{CP}^N$ using sections of $L_i^{k_1}$, and the limit $X$ is homeomorphic to a normal projective variety in $\mathbb{CP}^N$. Taking a subsequence and applying suitable projective transformations, the $M_i \subset \mathbb{CP}^N$ converge to $X$ as algebraic varieties.
\end{theorem}

This result implies that, given bounds on $n,d,v$, only finitely many Hilbert polynomials appear. An immediate corollary is the following.
\begin{cor}
  Given $n, d, v > 0$, there are finitely many diffeomorphism types of polarized K\"ahler manifolds $(M^n, L, \omega)$, of dimension $n$, such that $\mathrm{Ric}(\omega) > -\omega$, the curvature of $L$ is $-\sqrt{-1}\omega$, and $\mathrm{vol}(M) > v, \mathrm{diam}(M) < d$.
\end{cor}

The strategy of proof of Theorem~\ref{thm:main} follows Donaldson-Sun~\cite{DS1}, and a key step is the proof of the following partial $C^0$-estimate, conjectured originally by Tian~\cite{Tian90} for Fano manifolds.

\begin{theorem}\label{thm:partialc0}
  Given $n,d, v > 0$, there are $k_2, b > 0$ with the following property. Suppose that $(M, L, \omega)$ is a polarized K\"ahler manifold with $\mathrm{Ric}(\omega) > -\omega$, $\mathrm{vol}(M) > v$ and $\mathrm{diam}(M, \omega) < d$. Then for all $p\in M$, the line bundle $L^{k_2}$ admits a holomorphic section $s$ over $M$ satisfying $\vert s\Vert_{L^2} = 1$, and $|s(p)| > b$.
\end{theorem}
Tian~\cite{Tian90} conjectured this result under a positive lower bound for the Ricci curvature, with $L=K_M^{-1}$, and proved it in the two-dimensional case~\cite{Tian90a}. Donaldson-Sun~\cite{DS1} showed the result with two-sided Ricci curvature bounds, but arbitrary polarizations, and later several extensions of their result were obtained (see e.g.  \cite{CDS2,CDS3,Jiang16,Sz,PSS}). The result assuming a lower bound for the Ricci curvature, with $L=K_M^{-1}$ was finally shown by Chen-Wang~\cite{CW}. The improvement in our result is that we allow for general polarizations.

\medskip

Our next result addresses the structure of the singular set of the limit space $X$. In the setting of Theorem~\ref{thm:main}, if the metrics along the sequence are K\"ahler-Einstein, Donaldson-Sun~\cite{DS1} showed that the metric singular set of $X$ is the same as the complex analytic singular set of the corresponding projective variety (see also Corollary~\ref{cor:DSsingular}). In our setting this is not necessarily the case, however we have the following.

\begin{theorem}\label{thm:singularset}
  Let $(X,d)$ be a Gromov-Hausdorff limit as in Theorem~\ref{thm:main}.
 Then for any $\epsilon>0$, $X\setminus \mathcal{R}_\epsilon$ is contained in a finite union of analytic subvarieties of $X$. Furthermore, the singular set $X\setminus \mathcal{R}$ is equal to a countable union of subvarieties.
\end{theorem}

The ``almost regular'' set $\mathcal{R}_\epsilon\subset X$ is defined to be the set of points $p$ satisfying $\lim\limits_{r\to 0} r^{-2n}\mathrm{Vol}(B(p, r))  > \omega_{2n} - \epsilon$ in terms of the volume $\omega_{2n}$ of the unit ball in $\mathbb{C}^n$. The regular set is then $\mathcal{R} = \cap_{\epsilon > 0} \mathcal{R}_\epsilon$. Note that $\mathcal{R}_\epsilon$ is an open set, while in general $\mathcal{R}$ may not be open.

 Cheeger-Colding~\cite{[CC2]} showed that even in the Riemannian setting the Hausdorff codimension of $X \setminus\mathcal{R}$ is at least $2$, with more quantitative estimates obtained by Cheeger-Naber~\cite{CN}. Moreover, in a recent deep work of Cheeger-Jiang-Naber~\cite{CJN}, it was shown that for small $\epsilon$ the set $X\setminus \mathcal{R}_\epsilon$ has bounded $(2n-2)$-dimensional Minkowski content and is $2n-2$ rectifiable.
These results show that the singular set behaves well from the perspective of geometric measure theory.
On the other hand, the topology of the singular set could be rather complicated. In a recent paper of Li and Naber \cite{LN} (see also example $3.2$ of \cite{CJN}), it was shown that even assuming non-negative sectional curvature, non-collapsed limit spaces can have singular sets that are Cantor sets. Our Theorem~\ref{thm:singularset} shows that in sharp contrast with this, in the polarized K\"ahler setting the singular set has strong rigidity properties. For example if we perturb the K\"ahler metrics along our sequence locally inside a holomorphic chart and assume that the geometric assumptions are preserved, then the metric singular set can change by at most a countable set of points.

\medskip

A basic technical ingredient in this work is the following result, which is of independent interest.
\begin{theorem}\label{thm:chart}
  There exists $\epsilon > 0$, depending on the dimension $n$, with the following property. Suppose that $B(p,\epsilon^{-1})$ is a relatively compact ball in a (not necessarily complete)  K\"ahler manifold $(M^n, p, \omega)$, satisfying $\mathrm{Ric}(\omega) > -\epsilon \omega$, and
  \[ d_{GH}\Big( B(p,\epsilon^{-1}), B_{\mathbb{C}^n}(0,\epsilon^{-1})\Big) < \epsilon, \]
where $d_{GH}$ is the Gromov-Hausdorff distance.
 Then there is a holomorphic chart $F : B(p,1) \to \mathbb{C}^n$ which is a $\Psi(\epsilon|n)$-Gromov-Hausdorff approximation to its image. In addition on $B(p,1)$ we can write $\omega = i\partial\bar\partial\phi$ with $|\phi - r^2| < \Psi(\epsilon|n)$, where $r$ is the distance from $p$.
\end{theorem}
Here, and throughout the paper, $\Psi(\epsilon_1,\ldots, \epsilon_k| a_1,\ldots, a_l)$ denotes a function such that for fixed $a_i$ we have $\lim_{\epsilon_1, \ldots, \epsilon_k\to 0} \Psi = 0$. This result is an extension of    Proposition $1.3$ of \cite{L2}, where the bisectional curvature lower bound was assumed. See also \cite[Proposition 1]{CDS3}, where a similar conclusion is found under additional assumptions.
A simple consequence of the result is that for any non-collapsed Gromov-Hausdorff limit of K\"ahler manifolds with Ricci curvature bounded below, the set $\mathcal{R}_\epsilon$ defined above has the structure of a complex manifold, for sufficiently small $\epsilon$.

We give some further applications of this result. The first, Proposition~\ref{prop:scalar}, shows that under Gromov-Hausdorff convergence to a smooth K\"ahler manifold, the scalar curvature functions converge as measures. Here we state a simple corollary of this.
\begin{cor}\label{cor:Sintsmall}
  Given any $\epsilon > 0$, there is a $\delta > 0$ depending on $\epsilon, n$, satisfying the following. Let $B(p,1)$ be a relatively compact unit ball in a K\"ahler manifold $(M,\omega)$ satisfying $\mathrm{Ric} > -1$, and $d_{GH}(B(p,1), B_{\mathbb{C}^n}(0,1)) < \delta$. Then $\int_{B(p,1/2)} S < \epsilon$, where $S$ is the scalar curvature of $\omega$.
\end{cor}

When the manifold is polarized, non-collapsed, with Ricci curvature bounded below, then we obtain the following integral bound for the scalar curvature on any unit ball.
\begin{prop}\label{prop:Sintbdd}
 Let $B(p,1)$ be a unit ball in a polarized K\"ahler manifold $(M^n, L, \omega)$, satisfying $\mathrm{Ric} > -1$, and $\mathrm{vol}(B(p,1)) > v > 0$. There is a constant $C(n,v)$ depending on $n, v$ such that $\int_{B(p,1)} S < C(n,v)$.
\end{prop}
This is closely related to a question posed by Yau~\cite[Problem 9. p. 278]{Yau92}, on bounding the integral of scalar curvature on Riemannian manifolds with non-negative Ricci curvature. An argument similar to \cite[Proposition 2.7]{L4} shows that under a bisectional curvature lower bound the same result holds even without the polarization condition. See also Petrunin~\cite{Pet08} for an analogous result, where the sectional curvature is assumed to be bounded below, but non-collapsing is not required.

The final application is the following, which was proved previously by the first author~\cite{L2} under the assumption of non-negative bisectional curvature.
\begin{prop}\label{prop:Cn}
  There exists $\epsilon>0$ depending on $n$, so that if $M^n$ is a complete noncompact K\"ahler manifold with $\mathrm{Ric}\geq 0$ and $\lim\limits_{r\to\infty} r^{-2n}\mathrm{vol}(B(p, r))\geq \omega_{2n}-\epsilon$, then $M$ is biholomophic to $\mathbb{C}^n$. Here $\omega_{2n}$ is the volume of the Euclidean unit ball.
\end{prop}

We conclude this introduction with a brief description of the contents of the paper. In Section~\ref{sec:regular} we prove Theorem~\ref{thm:chart} and the two applications mentioned above. We then use the charts provided by Theorem~\ref{thm:chart} in Section~\ref{sec:partialc0} to construct global holomorphic sections of high powers of our line bundles, following the approach of Donaldson-Sun~\cite{DS1}. This leads to the partial $C^0$ estimate, and the proof of Theorem~\ref{thm:main}. In Section~\ref{sec:singularities} we study the relation between the complex analytic and metric singularities of $X$, proving Theorem~\ref{thm:singularset}.  The argument in Section~\ref{sec:partialc0} uses the recent estimates of Cheeger-Jiang-Naber~\cite{CJN}, but we show in the Appendix that our results can be obtained independently of \cite{CJN} by following the approach of Chen-Donaldson-Sun~\cite{CDS2}. In addition we prove a splitting result in the Appendix which is well known to experts but does not seem to be written up in the setting that we use.

\medskip

\begin{center}
\bf  {\quad Acknowledgments}
\end{center}
We would like to thank Jeff Cheeger and John Lott for their interest in this work. We also thank Aaron Naber for many fruitful discussions, as well as Richard Bamler and Peter Topping for helpful suggestions. Special thanks also goes to Yum-Tong Siu for the proof of Lemma \ref{SiuLemma}.

\section{Holomorphic charts near regular points}\label{sec:regular}
Our main goal in this section is to prove Theorem~\ref{thm:chart}.
We will first construct a holomorphic chart by regularizing the metric using Perelman's pseudolocality~\cite{Per} along the Ricci flow. The following is the basic input about the Ricci flow that we need.

\begin{prop}\label{prop:Ricflow1} There is a constant $D > 0$ such that given $a > 0$, for sufficiently small $\epsilon > 0$ the following holds. Let $B(p,\epsilon^{-1})$ be a relatively compact ball in a K\"ahler manifold $(M^n,g)$ satisfying $\mathrm{Ric}(g) > -\epsilon g$, such that
  \[d_{GH}(B(p,\epsilon^{-1}), B_{\mathbb{C}^n}(0,\epsilon^{-1})) < \epsilon.\]
Then there is a Ricci flow solution $g(t)$ on $B(p,\epsilon^{-1/2})$ for $t\in [0,1]$ with $g(0)=g$, such that
  \begin{itemize}
  \item On the ball $B_{g(1)}(p,10D)$, in suitable coordinates, we have
    \[ |g(1) - g_{Euc}|_{C^5(g_{Euc})} < \Psi(\epsilon|n); \]
    \item The curvature along the flow satisfies $|\mathrm{Rm}|< a/t$;
    \item We have the following estimates for the distance function along the flow:
      \[ \begin{aligned} d_t(x,y) &> d_0(x,y) - D\sqrt{t}, \\
          d_t(x,y) &< D(d_0(x,y) +  \sqrt{t}),
        \end{aligned}\]
      for $x,y\in B(p,\epsilon^{-1/2}/2)$ and $t \in [0,1]$.
    \end{itemize}
  \end{prop}
  \begin{proof}
    According to Cavalletti-Mondino~\cite{CM17}, our assumptions imply that the $\Psi(\epsilon|n)$-almost Euclidean isoperimetric inequality holds in balls of radius $\epsilon^{-1/2}$ inside the ball $B(p,\epsilon^{-1})$. We can then apply He~\cite[Lemma 2.4]{He17} (see also Hochard~\cite{Hoch} and Topping~\cite{To}) to conformally scale the metric $g$ on a domain $U\subset B(p,\epsilon^{-1})$ to a complete Riemannian metric $(U, h)$, such that $g=h$ on $B(p, \epsilon^{-1/2})$, the $\Psi(\epsilon|n)$-almost Euclidean isoperimetric inequality holds on balls of radius $\epsilon^{-1/2}/8$, and we have a lower bound $S_h \geq -\Psi(\epsilon|n)$ for the scalar curvature of $h$. As in \cite{He17} there exists a Ricci flow solution $h(t)$ for a definite time $t\in [0,T]$, satisfying $|\mathrm{Rm}| \leq A/t$ for $t\in (0,T]$. We can choose $A$ arbitrarily small, and $T$ as large as we like if $\epsilon$ is sufficiently small. The distance estimates follow from Lemma 3.5 and Lemma 3.7 in \cite{He17}.

To see the claim about comparing $g(t)$ with the Euclidean metric, suppose we have a sequence of such flows, with $\epsilon_i \to 0$. The curvature estimates, and the non-collapsing estimate (see Lemma 3.1 in \cite{He17}) applied for large times, imply that in the limit we end up with a (stationary) flow of flat metrics on $\mathbb{R}^{2n}$, and the claim follows from this.
\end{proof}

We need the following estimate along a Ricci flow as above, similar to Lemma 2.3 in Huang-Tam~\cite{HT15}.
\begin{lemma}\label{lem:HT}
  Suppose that we have a Ricci flow on $B(p,\epsilon^{-1/2})$ as given by Proposition~\ref{prop:Ricflow1}, for sufficiently small $\epsilon > 0$. Given constants $A, A_k, l > 0$ there are $C_k > 0$ satisfying the following. Let $f\geq 0$ be a smooth function on $B(p,\epsilon^{-1/2})\times [0,1]$, such that
  \begin{itemize}
  \item for all $k > 0$ $f$ satisfies
    \[ (\partial_t - \Delta) f(x,t) \leq \frac{A}{t}\,\max_{0 \leq s \leq t} f(x, s) + A_kt^k, \]
    on $B(p,\epsilon^{-1/2})\times(0,1)$.
  \item $\partial_t^k f |_{t=0} = 0$ for all $k\geq 0$ on $B(p,\epsilon^{-1/2})$.
  \item $\sup_{x\in B(p,\epsilon^{-1/2})} f(x,t) \leq A t^{-l}$ for $t\in (0,1]$.
  \end{itemize}
    Then we have $\sup_{x\in B(p,1)} f(x,t) \leq C_kt^k$ for all $k \geq 0$, and $t\in [0, 1]$.
  \end{lemma}
  \begin{proof}
    The proof is similar to the first part of the proof of Huang-Tam~\cite{HT15}, Lemma 2.3.
We set $\phi(s)$ to be a cutoff function such that $\phi(s)  = 0$ for $s \geq 3/4$, and $\phi(s) = 1$ for $s \leq 1/4$. Let $\Phi = \phi^m$ for suitable $m$ to be chosen later, and we set $q = 1-2/m$. This satisfies
    \[ 0 \leq \Phi \leq 1, \quad -C_m \Phi^q \leq \Phi' \leq 0, \quad |\Phi''| \leq C_m \Phi^q, \]
    for a constant $C_m$ depending on $m$.  Let us define $\rho:B_{g(1)}(p,9D) \to \mathbb{R}$ to be given by $\rho(x) = d_{g(1)}(p,x)^2/(9D)^2$. Then $|\nabla_{g(1)}\rho|, |\Delta_{g(1)}\rho| < C$, and so as in Lemma 2.2 in Huang-Tam~\cite{HT15}, we have
$ |\nabla_{g(t)}\rho| < Ct^{-ca},  |\Delta_{g(t)} \rho| < Ct^{-1/2 - ca}$
    for a dimensional constant $c$ (the constant $a$ is the constant in the estimate $|\mathrm{Rm}| < a/t$). We choose $\epsilon$ sufficiently small, so that $a$ satisfies $ca < 1/4$. Then
    \[ |\nabla_{g(t)}\rho| < Ct^{-1/4}, \quad |\Delta_{g(t)} \rho| < Ct^{-3/4}. \]

  We then set $\Psi(x) = \Phi(\rho(x))$, so that by definition $\Psi$ vanishes outside of $B_{g(1)}(p,9D)$ for all $t$. We also let $\theta(t) = \mathrm{exp}(-\alpha t^{1-\beta})$ for $\alpha > 0, \beta\in (0,1)$. For any $K > A$, let $F = ft^{-K}$. Then we have a constant $C$ (which may change from line to line) such that
 \[ \begin{aligned}
     (\partial_t - \Delta) F(x,t) &\leq -K t^{-1-K}f(x,t)  + \frac{A}{t} \max_{s\leq t} t^{-K}f(x,s)  + A_kt^{k-K}\\
     &\leq -\frac{A}{t} F(x,t) + \frac{A}{t} \max_{s\leq t} F(x,s) + C,
\end{aligned}
   \]
    and in addition $F \leq A t^{-l-K}$. We will show that the smooth function
    \[ H(x,t) = \theta(t) \Psi(x) F(x,t), \]
    is a priori bounded on $B(p,\epsilon^{-1/2})\times [0,1]$.
    Suppose that $H$ achieves its maximum at a point $(x_0, t_0)$.

    At the maximum we have $\nabla H = 0$, therefore $\Psi\nabla F + F\nabla\Psi = 0$, and so
    \[ \nabla\Psi\cdot \nabla F = -\frac{F |\nabla \Psi|^2}{\Psi}. \]
    Note that by the estimates above we have
    \[ |\Delta\Psi| = \big|\Phi''(\rho) |\nabla\rho|^2 + \Phi'(\rho) \Delta\rho \big| \leq C_m\Psi^q t_0^{-3/4}, \]
    and also
    \[ \frac{|\nabla\Psi|^2}{\Psi} \leq C_m\Psi^{2q-1} t_0^{-1/2}. \]
    In addition, note that by the maximality of $H$ at $(x_0,t_0)$, for any $s\leq t_0$ we have
    \[ \theta(s)F(x_0,s) \leq \theta(t_0)F(x_0, t_0), \]
    and so since $\theta$ is decreasing, we have
    $\max_{s\leq t_0} F(x_0,s) \leq F(x_0, t_0)$. It follows that at $(x_0, t_0)$
    \[ (\partial_t - \Delta) F \leq C. \]

    At the maximum we can then compute
    \[ \begin{aligned}  \Delta H &= \theta F \Delta\Psi  + \theta\Psi \Delta F + 2\theta \nabla\Psi\cdot \nabla F \\
        &\geq -C_m \theta F\Psi^q t_0^{-3/4} -C_m \theta F \Psi^{2q-1} t_0^{-1/2} + \theta\Psi \Delta F,
      \end{aligned} \]
    and
    \[ \partial_t H = \theta\big( -\alpha(1-\beta)t_0^{-\beta} \Psi F + \Psi \partial_t F\big). \]
    It follows that
    \[ \begin{aligned}
        0 &\leq (\partial_t - \Delta) H \\
        &\leq  \theta\Psi (\partial_t - \Delta)F + \theta\big[ -\alpha(1-\beta) t_0^{-\beta} \Psi F + C_m \Psi^q F t_0^{-3/4} + C_m \Psi^{2q-1} F t_0^{-1/2}\big] \\
        &\leq \theta\big[ C\Psi -\alpha(1-\beta) t_0^{-\beta} \Psi F \\
        &\qquad\qquad + C_m (\Psi F)^q  t_0^{-3/4-(1-q)(l+K)} + C_m (\Psi F)^{2q-1}  t_0^{-1/2-(2-2q)(l+K)}\big],
      \end{aligned} \]
    using that $F \leq A t^{-l-K}$.

    If we choose $q$ very close to 1 (i.e. $m$ very large), then we can choose $\beta \in (0,1)$ so that $\beta-3/4-(1-q)(l+K) > 0$ and $\beta-1/2-(2-2q)(l+K) > 0$. Then our previous inequality implies (multiplying through by $t_0^\beta$), that
    \[ 0 \leq -\alpha(1-\beta) \Psi F + C\Psi + C_m\big[ (\Psi F)^q + (\Psi F)^{2q-1}\big]. \]
    We can now choose $\alpha$ so large that $\alpha(1-\beta) > 2C_m + 1$. Then
    \[ (\Psi F) \leq C\Psi + C_m\big[ (\Psi F)^q + (\Psi F)^{2q-1} - 2(\Psi F)\big]. \]
    It follows that if $\Psi F$ is sufficiently large, then this leads to a contradiction, and so $\Psi F \leq C$. This implies that $H\leq C$ at the maximum, as required. This implies the estimate $F\leq C$ on $B_{g(1)}(p, 8D)$, and the bounds for the distance along our Ricci flow imply that $B_{g(0)}(p, 1)\subset B_{g(1)}(p, 8D)$.
\end{proof}

  \begin{prop}\label{prop:chart1} Suppose that we are in the situation of Proposition~\ref{prop:Ricflow1}, with sufficiently small $\epsilon$. Then on a smaller ball $B(p,r)$ we have a holomorphic chart $F : B(p,r) \to \mathbb{C}^n$, such that for suitable $r_1, r_2 > 0$ the image of $F$ satisfies
    \[ B(0,r_1) \subset F(B(p,r)) \subset B(0,r_2). \]
\end{prop}
\begin{proof}
 We consider the Ricci flow $g(t)$ on $B = B(p,\epsilon^{-1/2})$ provided by Proposition~\ref{prop:Ricflow1}, and show that the metric $g(1)$ is ``approximately K\"ahler'' in the sense that on a smaller ball, $B(p,r_0)$, we have $|\nabla^iJ_0| < C$ for $i\leq 5$. Here we denote by $J_0$ the fixed complex structure on $B(p,\epsilon^{-1/2})$ to distinguish it from another, time dependent family of almost complex structures $J(t)$ below. We follow the argument given by Kotschwar~\cite{Kot17} for preserving the K\"ahler condition along a Ricci flow, using Lemma~\ref{lem:HT} (see also Huang-Tam~\cite{HT15}, Shi~\cite{Shi97}).

 Following Kotschwar~\cite{Kot17} we first define a family $J(t)$ of almost complex structures by $J(0) = J_0$, and
 \[ \frac{\partial}{\partial t} J^a_b = R^c_b J^a_c - R^a_c J^c_b. \]
 It is convenient to introduce a differential operator $D_t$, in terms of which $D_t J = 0$ (see \cite{Kot17} for details). We also have $D_t g = 0$, and $g$ remains Hermitian for the almost complex structure $J$.

Define $F, G \in \mathrm{End}(\wedge^2 T^*B)$ by
 \[ (F\eta) (X,Y) = \frac{1}{2}(\eta(X, JY) + \eta(JX,Y)), \quad (G\eta)(X,Y) = \eta(JX, JY). \]
 We then let
 \[ \tilde{P} = \frac{1}{2} (\mathrm{Id} + G), \quad \hat{P} = \frac{1}{2}(\mathrm{Id} - G). \]
 On the complexification $\wedge^2_{\mathbb{C}} T^*B$ we have
 \[ P^{(2,0)} = \frac{1}{2} \left(\hat{P} - \sqrt{-1}F\right), \quad P^{(1,1)} = \tilde{P}, \quad P^{(0,2)} = \overline{P^{(2,0)}}, \]
 for the orthogonal projection maps onto $\wedge^{2,0}B, \wedge^{1,1}B, \wedge^{0,2}B$. We will control the complex structure $J$ along the flow in terms of the piece $\hat{R} = R\circ \hat{P}$, where $R : \wedge^2T^*B \to \wedge^2 T^*B$ is the curvature operator of $g(t)$. The derivatives of $J$ will then be controlled by derivatives $\nabla^i \hat{R}$. The evolution equations of these in turn depend on derivatives $\nabla^i \hat{P}$. Using that the initial metric is K\"ahler, at $t=0$ we have $\partial_t^k\nabla^i \hat{R} = 0$ and $\partial_t^{k+1}\nabla^i \hat{P} = 0$ for all $k\geq 0$. For the calculations we also have the commutation formulas (see Kotschwar~\cite[Lemma 4.3]{Kot11}) for a tensor $X$
 \[ \begin{aligned}
     \, [D_t, \nabla] X &= \nabla R \ast X + R \ast \nabla X \\
     [D_t - \Delta, \nabla_a] &= 2R_{abdc}\Lambda^c_d \nabla_b + 2R_{ab}\nabla_b,
   \end{aligned} \]
 where $\Lambda$ is a certain algebraic operation on tensors.

 Let us write $\hat{S} = (\nabla R) \circ \hat{P}$  and $\hat{T} = (\nabla^2R) \circ\hat{P}$. Here the action of $\hat{P}$ is such that for instance $\hat{S}(X, \eta) = (\nabla_X R)\hat{P}(\eta)$ for a vector $X$ and 2-form $\eta$. Note that $\hat{S} = \nabla\hat{R} + R\ast \nabla\hat{P}$, and $\hat{T} = \nabla\hat{S} + \nabla R \ast \nabla\hat{P}$. The advantage of $\hat{S}$ and $\hat{T}$ over $\nabla\hat{R}, \nabla^2\hat{R}$ is that their evolution equations only involve up to two derivatives of $\hat{P}$, and their norms are controlled by $|\nabla R|, |\nabla^2 R|$ since $\hat{P}$ is a projection.
 From \cite[Proposition 4.5]{Kot11} we have
 \[ D_t \nabla\hat{P} = R\ast \nabla\hat{P} + \hat{P} \ast \hat{S}, \]
 and so we also have
 \[ \begin{aligned}
     D_t \nabla^2\hat{P} &= \nabla D_t\nabla\hat{P} + \nabla R\ast \nabla\hat{P} + R\ast \nabla^2\hat{P} \\
     &= R\ast \nabla^2\hat{P} + \nabla R \ast \nabla\hat{P} + \nabla\hat{P} \ast \hat{S} + \hat{P} \ast \hat{T} + \hat{P} \ast R \ast \nabla\hat{P}.
   \end{aligned} \]
 Using the estimates for $\nabla^i R$ along the flow, it follows that
 \begin{equation}\label{eq:dthatP} \begin{aligned}
     \partial_t |\nabla\hat{P}| &\leq \frac{A}{t} |\nabla\hat{P}| + c|\hat{S}| \\
     \partial_t |\nabla^2\hat{P}| &\leq \frac{A}{t} |\nabla^2\hat{P}| + \frac{A}{t^{3/2}}|\nabla \hat{P}| + c|\hat{T}|,
   \end{aligned} \end{equation}
 for a dimensional constant $c$.

 For the evolution of the curvature we have (see \cite[Proposition 4.7, Lemma 4.9]{Kot11})
 \[ \begin{aligned}
     (D_t - \Delta) R &= R\ast R,  \\
     [(D_t - \Delta)R] \circ \hat{P} &= R\ast \hat{R} + \hat{P} \ast R\ast \hat{R}.
   \end{aligned}\]
 It follows that
 \[ \begin{aligned}
     (D_t - \Delta) \hat{R} &= [(D_t - \Delta) R] \circ \hat{P} + R \ast \nabla^2\hat{P} + \nabla R \ast \nabla\hat{P}  \\
     &= R\ast \hat{R} + \hat{P} \ast R\ast \hat{R} + R\ast \nabla^2\hat{P} + \nabla R\ast\nabla\hat{P},
   \end{aligned} \]
 and so
 \begin{equation}\label{eq:dthatR}
   (\partial_t - \Delta) |\hat{R}| \leq \frac{A}{t} |\hat{R}| + \frac{A}{t} |\nabla^2\hat{P}| + \frac{A}{t^{3/2}} |\nabla\hat{P}|.
 \end{equation}
 Similarly we have
 \[     (D_t - \Delta) \hat{S} = [(D_t - \Delta) \nabla R] \circ \hat{P} + \nabla R\ast \nabla^2\hat{P} + \nabla^2 R\ast \nabla\hat{P}.\]
 Using the commutation relations,
 \[ \begin{aligned}
     \, [(D_t - \Delta) \nabla_aR] \circ \hat{P} &= [ \nabla(D_t - \Delta)R ]\circ \hat{P} + [2R_{abdc}\Lambda^c_d\nabla_bR + 2R_{ab}\nabla_b R]\circ \hat{P}
   \end{aligned} \]
 For the term involving $\Lambda_d^c$ we have (see the calculation in \cite[Proposition 4.13]{Kot11})
 \[ \hat{P}_{ijkl} R_{abdc}\Lambda_d^c\nabla_bR_{klmn} = R\ast \hat{S} + \nabla R\ast \hat{R} \ast \hat{P}. \]
It follows that
\[
\begin{aligned}
  \, [(D_t - \Delta) &\nabla_aR] \circ \hat{P} = \nabla\big( [(D_t - \Delta)R]\circ \hat{P}\big) + [(D_t - \Delta)R] \ast \nabla\hat{P} \\
  &\quad + R\ast \hat{S} + \nabla R\ast\hat{R}\ast\hat{P}  \\
     &= \nabla R\ast \hat{R} + R\ast \nabla\hat{R} + \nabla\hat{P}\ast R \ast \hat{R} + \hat{P}\ast \nabla R \ast \hat{R} + \hat{P}\ast R\ast \nabla\hat{R} \\
     &\quad + R\ast R \ast \nabla\hat{P} + R\ast \hat{S} \\
     &= R\ast \hat{S} + \nabla R \ast \hat{R} + R\ast R\ast\nabla\hat{P} + R\ast\hat{R}\ast\nabla\hat{P} + \hat{P}\ast\nabla R\ast \hat{R} \\
     &\quad + \hat{P}\ast R\ast \hat{S} + \hat{P}\ast R\ast R\ast \nabla\hat{P}.
   \end{aligned} \]
This implies
 \begin{equation}\label{eq:dthatS} (\partial_t - \Delta) |\hat{S}| \leq \frac{A}{t} |\hat{S}| + \frac{A}{t^{3/2}}|\nabla^2 \hat{P}| + \frac{A}{t^2} |\nabla\hat{P}| + \frac{A}{t^{3/2}}|\hat{R}|. \end{equation}
 For $\hat{T}$, we have
\[
     (D_t - \Delta) \hat{T} = [(D_t - \Delta) \nabla^2 R] \circ \hat{P} + \nabla^2 R \ast \nabla^2\hat{P} + \nabla^3 R \ast \nabla\hat{P}.
\]
By the commutation relations again
\[ \begin{aligned} \, [(D_t - \Delta)\nabla_a \nabla R] \circ \hat{P} &=
    [ \nabla(D_t -\Delta) \nabla R] \circ\hat{P} \\ &\quad + [2R_{abdc}\Lambda_d^c\nabla_b\nabla R + 2R_{ab}\nabla_b \nabla R] \circ\hat{P},
  \end{aligned} \]
By the same argument as \cite[Proposition 4.13]{Kot11}, the term involving $\Lambda$ satisfies
\[ R_{abdc}\Lambda_d^c\nabla_b\nabla R \circ\hat{P} = R \ast \hat{T} + \nabla^2R \ast \hat{R} \ast \hat{P}. \]
At the same time, we have
\[ \begin{aligned}
   \, [\nabla (D_t - \Delta)\nabla R] \circ\hat{P} &= \nabla\big([(D_t - \Delta)\nabla R]\circ \hat{P}\big) + [(D_t - \Delta)\nabla R] \ast \nabla\hat{P}.
  \end{aligned} \]
Using the calculations above, and collecting various terms, we obtain
\begin{equation}\label{eq:dthatT}
    \begin{aligned}
    (\partial_t - \Delta) |\hat{T}| &\leq \frac{A}{t}|\hat{T}| + \frac{A}{t^2}|\nabla^2\hat{P}| + \frac{A}{t^{5/2}}|\nabla\hat{P}| + \frac{A}{t^2}|\hat{R}| + \frac{A}{t^{3/2}}|\hat{S}|.
  \end{aligned}
\end{equation}

Given a constant $K > 0$, using Equations~\eqref{eq:dthatP}, \eqref{eq:dthatR}, \eqref{eq:dthatS}, \eqref{eq:dthatT}, we have
\[ \begin{aligned}
    \partial_t t^{-K-1/2}|\nabla\hat{P}| &\leq \frac{A-K-1/2}{t} t^{-K-1/2}|\nabla\hat{P}| + c t^{-K-1/2}|\hat{S}| \\
    \partial_t t^{-K}|\nabla^2\hat{P}| &\leq \frac{A-K}{t} t^{-K}|\nabla^2\hat{P}| + \frac{A}{t} t^{-K-1/2}|\nabla\hat{P}| + c t^{-K}|\hat{T}| \\
    (\partial_t - \Delta) t^{-K-1} |\hat{R}| &\leq \frac{A-K-1}{t} t^{-K-1}|\hat{R}| + \frac{A}{t^2} t^{-K}|\nabla^2\hat{P}| + \frac{A}{t^2} t^{-K-1/2} |\nabla\hat{P}| \\
    (\partial_t - \Delta) t^{-K-1/2} |\hat{S}| &\leq \frac{A-K-1/2}{t} t^{-K-1/2} |\hat{S}| + \frac{A}{t^2} t^{-K} |\nabla^2\hat{P}| \\ &\quad + \frac{A}{t^2} t^{-K-1/2} |\nabla\hat{P}| + \frac{A}{t} t^{-K-1} |\hat{R}| \\
    (\partial_t - \Delta) t^{-K} |\hat{T}| &\leq \frac{A-K}{t} t^{-K}|\hat{T}| + \frac{A}{t^2} t^{-K} |\nabla^2\hat{P}| + \frac{A}{t^2} t^{-K-1/2} |\nabla\hat{P}| \\ &\quad + \frac{A}{t} t^{-K-1}|\hat{R}| + \frac{A}{t} t^{-K-1/2} |\hat{S}|
 \end{aligned} \]

Let $K > 3A+2$ and define
 \[  \begin{aligned}
     Y &= t^{-K}|\nabla^2 \hat{P}| + t^{-K-1/2} |\nabla\hat{P}| \\
     Z &= t^{-K}|\hat{T}| + t^{-K-1/2} |\hat{S}| + t^{-K-1} |\hat{R}|.
   \end{aligned} \]
 From the inequalities above we then obtain
 \[ \begin{aligned}
     \partial_t Y &\leq c Z \\
     (\partial_t - \Delta) Z &\leq \frac{A}{t^2} Y.
   \end{aligned} \]
 Note that $Y$ is still smooth up to $t=0$, and $Y(x,0) = 0$ for all $x$. It follows that
 \[ Y(x,t) \leq t c\max_{s\leq t} Z(x,s), \]
 and therefore $Z$ satisfies the inequality (for a larger choice of $A$)
 \[ (\partial_t - \Delta) Z(x,t) \leq \frac{A}{t} \max_{s\leq t} Z(x,s). \]
 All $t$-derivatives of $Z$ vanish at $t=0$ because the initial metric is K\"ahler. In addition since $\hat{P}$ is a projection map, the norms $|\hat{R}|, |\hat{S}|, |\hat{T}|$ are controlled by $|R|, |\nabla R|, |\nabla^2 R|$. By the curvature estimates along the flow we have $Z \leq A / t^l$ for some $A, l$. We can therefore apply Lemma~\ref{lem:HT} to obtain that $Z \leq C_k t^k$ on $B(p, 1)$. In turn this also implies that $Y \leq C_k t^k$.

 We can apply the same argument to obtain estimates for further derivatives of the curvature composed with $\hat{P}$, inductively. Analogously to the above, we have inequalities
 \[ (\partial_t - \Delta) |\nabla^i R \circ \hat{P}| \leq \frac{A}{t} |\nabla^i R\circ \hat{P}| + C_kt^k, \]
 where we are using the inductive assumption to control $|\nabla^j R\circ \hat{P}|$ for $j < i$. It follows that on a smaller ball $B(p, r_0)$ we have $|\nabla^i R\circ \hat{P}| < C_kt^k$, for $i < 5$, say. This in turn implies estimates  $|\nabla^i \hat{P}| \leq C_kt^k$ for $0< i<5$.

 We can now use this to control $\nabla^i J$ for $0 < i < 5$. Indeed, the evolution of $\nabla J$ has the form (see \cite[Lemma 7]{Kot17})
 \[ D_t \nabla J = \hat{S} \ast J + R\ast \nabla\hat{P} \ast J + R\ast \nabla J, \]
 and so
 \[ \partial_t |\nabla J| \leq \frac{A}{t} |\nabla J| + C_kt^k. \]
 it follows that $|\nabla J| \leq C_kt^k$, since all $t$-derivatives of $\nabla J$ vanish at $t=0$. For the higher derivatives of $J$ we can use the commutation relation to get
 \[ \begin{aligned} D_t \nabla^i J &= \nabla D_t \nabla^{i-1}J + R\ast \nabla^{i-1} J + \nabla R \ast \nabla^i J
   \end{aligned} \]
 and so we can inductively find inequalities of the form
 \[ \partial_t |\nabla^i J| \leq \frac{A}{t} |\nabla^i J| + C_kt^k, \]
 for $0 < i < 5$, say. We therefore have $|\nabla^i J| \leq C_kt^k$ for $0 < i < 5$ on $B(p,r_0)$. From these bounds we find that the curvature endomorphism satisfies, for a local orthonormal frame $e_i$, that
 \[ |R(e_i, e_j) Je_k - J R(e_i, e_j) e_k| \leq C_kt^k, \]
 and from this it follows that the Ricci endomorphism $Rc$ also satisfies $|Rc(Je_i) - JRc(e_i)| \leq C_k t^k$. The evolution of $J$ is given by $\partial_t J = J\circ Rc - Rc\circ J$, and so we find that $|J - J_0| \leq C_kt^k$. Similarly we can also obtain $|\nabla^i\partial_t J| \leq C_kt^k$, and so we can inductively obtain estimates $|\nabla^i (J-J_0)| \leq C_kt^k$ for $i < 5$, using the equations
 \[ D_t  \nabla^i(J-J_0) = R\ast \nabla^i(J-J_0) + \nabla R \ast \nabla^{i-1}(J - J_0) + \nabla D_t \nabla^{i-1}(J-J_0). \]
Finally we can conclude that we have estimates $|\nabla^i J_0| \leq C_kt^k$ on $B(p,r_0)$ for $t \leq 1$.

By rescaling (and choosing $\epsilon$ smaller), we can assume that we have estimates $|\nabla^i J_0| \leq C$ with respect to the metric $g(1)$, on $B_g(p, 6D)$. By the properties of the flow $g(t)$ in Proposition~\ref{prop:Ricflow1} we can view $g(1)$ as a metric on the Euclidean ball $B_{\mathbb{C}^n}(0,5D)$, close in $C^5$ to the Euclidean metric, and so $J_0$ has bounded derivatives in terms of the Euclidean metric. After a linear change of coordinates with bounded eigenvalues, we can assume that $J_0$ is standard at the origin, and it still has bounded derivatives. We can now find holomorphic functions on a small ball $B_{g(1)}(p,r_0)$ which are perturbations of the complex linear functions, for instance by the approach of H\"ormander~\cite{Ho66} using $L^2$-estimates to the Newlander-Nirenberg theorem. If we choose $\epsilon$ sufficiently small, then by the distance estimates along the flow (applied for small $t$) we can obtain
\[ B_{g(1)}(p, r_0/3D) \subset B_{g(0)}(p, r_0/2D)\subset B_{g(1)}(p, r_0), \]
since by the curvature estimates, for any small $t_0 > 0$ we can assume that $g(t)$ is very close to $g(1)$ for $t\in [t_0, 1]$. Setting $r=r_0/2D$, we thus obtain a holomorphic chart as required.
\end{proof}

We next construct a bounded K\"ahler potential locally.

\begin{prop}\label{prop:potential1}
  There exist $\epsilon, C > 0$ with the following property. Suppose that $B(p,\epsilon^{-1})$ is relatively compact in a K\"ahler manifold $(M^n, \omega)$ satisfying $\mathrm{Ric}(\omega) > -\epsilon \omega$, and in addition  $d_{GH}(B(p, \epsilon^{-1}), B_{\mathbb{C}^n}(0, \epsilon^{-1}))$. Then on $B(p,C^{-1})$ we can write $\omega = i \partial\bar\partial \phi$ with $|\phi| < C$.
\end{prop}
\begin{proof}
  In Proposition~\ref{prop:chart1} we have constructed a holomorphic chart on a small ball around $p$, and by choosing $\epsilon$ smaller and scaling, we can assume that the chart
  \[ F : B(p,1) \to \mathbb{C}^n \]
is defined on $B(p,1)$, and $B(0, r_1)\subset \mathbb{C}^n$ is contained in its image. In terms of this chart, we can view our metric $\omega$ as defining a metric on $B(0,r_1)$. We first ``glue'' this metric onto $\mathbb{CP}^n$ in order to run the Ricci flow.

On $B(0,r_1)$ we can write $\omega = i\partial\bar\partial \psi$, where we can assume that $\psi \geq 0$. Consider the function
  \[ f(z) = \log(1 + |z|^2) - \log(1 + r_1^2/4), \]
  on $\mathbb{C}^n$, that is negative in $B(0,r_1/2)$ and positive elsewhere. It follows that for sufficiently large $K$, if we take a regularized maximum
  \[ h(z) = \widetilde{\mathrm{max}}\{ \psi, Kf \} \]
  then $\eta = i\partial\bar\partial h$ is a well defined metric on $\mathbb{C}^n$ extending to a metric on $\mathbb{CP}^n$, such that $\eta = \omega$ in $B(0,r_1/4)$. The K\"ahler class $[\eta] = K[\omega_{FS}]$, where $\omega_{FS}$ is the Fubini-Study metric. Note that since $\nabla F$ is bounded, we have $F(B(p,\delta_1))\subset B(0,r_1/4)$ for suitable $\delta_1 > 0$.

  It follows from Tian-Zhang~\cite{TZ}, that if $K$ is large (not necessarily bounded a priori), then we have a well defined K\"ahler-Ricci flow solution $\eta_t$ for $t\in [0,1]$. Since for the initial metric $B(p,\delta_1)$ is Gromov-Hausdorff close to the Euclidean ball, we can again apply the result of Cavalletti-Mondino~\cite{CM17} and the pseudolocality theorem (alternatively we could apply the version of the pseudolocality theorem proved by Tian-Wang~\cite{TW}). We find that if $\epsilon$ is sufficiently small, then for small $\delta_2, T > 0$ the metric $\eta_T$ on the ball $B_{\eta_T}(p, \delta_2)$ is $\Psi(\epsilon|n,T)$-close to the Euclidean metric in $C^5$ (and it is K\"ahler), therefore $\eta_T = i\partial\bar\partial \phi_T$ for a potential $\phi_T$ satisfying $|\phi_T| < C$. We will fix $T$ below, depending on the distance estimates following from the pseudolocality theorem.

Letting $\Omega$ be a fixed holomorphic volume form on $B_{\eta_T}(p, \delta_2)$, we can find a family of potentials $\phi_t$ for $\eta_t$ on $B_{\eta_T}(p, \delta_2)$ by solving
the equation
\begin{equation}\label{eq:ddtphit}
  \frac{d}{dt} \phi_t = \log \frac{\eta_t^n}{\Omega}.
\end{equation}
By pseudolocality we have the estimate $|\mathrm{Ric}(\eta_t)|_{\eta_t} < C/t$, and so the Ricci flow equation $\partial_t \eta_t = -\mathrm{Ric}(\eta_t)$ implies that the eigenvalues $\lambda_t$ of $\eta_t$ relative to $\eta_T$ satisfy $C^{-1} t < |\lambda_t| < C/t$. Therefore
\[ \left| \log \frac{\eta_t^n}{\Omega}\right| < C |\log t|, \]
and so from Equation~\eqref{eq:ddtphit}, together with the bound for $\phi_T$, we obtain $|\phi_0| < C$. Therefore the metric $\eta$ has a bounded K\"ahler potential on the ball $B_{\eta_T}(p, \delta_2)$. As long as $T$ is chosen sufficiently small, depending on the distance distortion estimates as in Proposition~\ref{prop:Ricflow1}, this ball contains the ball $B_{\eta}(p, \delta_3)$ for suitable $\delta_3$.
\end{proof}

We now prove Theorem~\ref{thm:chart}

\begin{proof}[Proof of Theorem~\ref{thm:chart}]
  Suppose that we have a sequence $\epsilon_i\to 0$, and corresponding balls $B(p_i, \epsilon_i^{-1})$ with metrics $g_i$ as in the statement of the proposition. From Proposition~\ref{prop:potential1}, we can assume that for large $i$, we have holomorphic charts $F_i: B(p_i,2)\to\mathbb{C}^n$, and K\"ahler potentials $\phi_i$ for $g_i$ on $B(p_i, 2)$, with $|\phi_i|<C$. In addition, by assumption, the $B(p_i,2)$ converge to $B_{\mathbb{C}^n}(0,2)$ in the Gromov-Hausdorff sense.

We can now argue along the lines of the proof of Proposition 3.1 in \cite{L2} to show that for sufficiently large $i$, on smaller balls $B(p_i, \delta)$ with $\delta=\delta(n) > 0$ we can find holomorphic charts $z^i_j$, which give a $\Psi(i^{-1})$-Gromov-Hausdorff approximation to the Euclidean ball $B(0,\delta)$. For this, note first that under the holomorphic charts $F_i$ above, we have $B(0, r_1)\subset F_i(B(p_i, 2))$ for some $r_1 > 0$. We can assume that $F_i(p_i)=0$ for all $i$. Let $U_i \subset B(p_i, 2)$ denote the connected component of $F_i^{-1}(B(p_i, r_1/2))$ containing $p_i$. Then $U_i$ is a Stein domain, which implies that we can apply the H\"ormander $L^2$ existence theorem (see Demailly~\cite{Dem} Theorem 5.1) on $U_i$ with the trivial line bundle equipped with metric $e^{-\phi_i}$. Note that the $\phi_i$ are uniformly bounded, and $B(p_i, r_2)\subset U_i$ for a fixed $r_2 >0$ by the Cheng-Yau gradient estimate for $F_i$. As in \cite{L2}, using Cheeger-Colding~\cite{CC} and Cheeger-Colding-Tian~\cite[Section 9]{CCT}, we have harmonic functions $w^i_j$ on $B(p_i, 2)$ which give a $\Psi(i^{-1})$-Gromov-Hausdorff approximation to $B(0,1)$, and in addition satisfy
\[ \dashint_{B(p_i,1)} |\bar\partial w^i_j|^2 < \Psi(i^{-1}). \]
Using the $L^2$-estimate we can perturb the $w^i_j$ on $U_i$ to holomorphic functions $z^i_j$, which still give a $\Psi(i^{-1})$-Gromov-Hausdorff approximation from $B(p_i, r_2)$ to their image in Euclidean space. Just as in \cite[Claim 3.2]{L2}, on a smaller ball $B(p_i, \delta)$ the $z^i_j$ define holomorphic charts for sufficiently large $i$.

Using our charts, we can now view the metrics $\omega_i$ as defining metrics on the Euclidean ball $B(0, \delta/2)$, with uniformly bounded potentials $\phi_i$ (i.e. we identify the functions $z^i_j$ with $z_j$). The identity map on $B(0,\delta/2)$ is then a $\Psi(i^{-1})$-Gromov-Hausdorff approximation from $\omega_i$ to the Euclidean metric $\omega_{Euc}$, while the gradient bound for the holomorphic functions $z_i$ with respect to $\omega_i$ implies that we have a lower bound $\omega_i > C^{-1}\omega_{Euc}$.
The $\phi_i$ satisfy $\Delta_{\omega_i} \phi_i = 2n$, and so using the gradient estimate, we can take a limit $\phi_\infty$ on $B(0, \delta/2)$. By the same argument as the proof of Claim~\ref{claim:pluriharmonic} below, the function $\phi_\infty - r^2$ is pluriharmonic.  It follows that $\widetilde{\phi}_i = \phi_i + (r^2 - \phi_\infty)$ are also K\"ahler potentials for $\omega_i$, and by construction $|\widetilde{\phi}_i - r_i^2| < \Psi(i^{-1})$, where $r_i$ is the $\omega_i$-distance from $p_i$ .
\end{proof}

Using such holomorphic charts, the following proposition allows us to define a complex structure on the almost regular set $\mathcal{R}_\epsilon$ in the limit space of K\"ahler manifolds with Ricci curvature bounded below, for sufficiently small $\epsilon$.

\begin{prop}\label{prop:limitchart} There exists $\epsilon = \epsilon(n)>0$ so that the following holds.
Let $(M_i^n, p_i, \omega_i)$ be a sequence of K\"ahler manifolds (not necessarily complete) so that $\mathrm{Ric}\geq -\epsilon$ and $B(p_i, \frac{2}{\epsilon})\subset\subset M_i$, with $d_{GH}(B(p_i, 2\epsilon^{-1}), B_{\mathbb{C}^n}(0, 2\epsilon^{-1})) < \epsilon$.
Assume that $(M_i^n, p_i)\to (X, p)$ in the pointed Gromov-Hausdorff sense. Let $(z^i_1, .., z^i_n)$ be holomorphic charts on $B(p_i, 10)$ obtained using Theorem~\ref{thm:chart}, and let us assume $z^i_j\to z_j$ on $B(p, 5)$.
Then $(z_1, .., z_n)$ is a homeomorphism from $B(p, 3)$ to the image.
\end{prop}
\begin{proof}
The argument is similar to Proposition $6.1$ of \cite{L1}, and the main point is to prove the injectivity. Suppose that $x_1, x_2\in B(p, 5)$, with $d(x_1, x_2) = 10d\neq 0$. We will show that the coordinates $(z_1, ..., z_n)$ separate them.
Consider sequences $M_i\ni x_1^i\to x_1, M_i\ni x_2^i\to x_2$. Then for large $i$, $B(x^i_1, d)\cap B(x^i_2, d)=\emptyset$. Using volume comparison, we see that $B(x^i_1, \epsilon_0^{-1} d)$ is $\epsilon_0d$-Gromov-Hausdorff close to a Euclidean ball if $\epsilon$ is sufficiently small, where $\epsilon_0$ is the parameter from Theorem~\ref{thm:chart}. It follows that we have holomorphic charts $(z^i_{k1}, ..., z^i_{kn})$ around $x_k$ $(k = 1, 2)$, which give Gromov-Hausdorff approximations from each $B(x^i_k, d)$ to the Euclidean ball. As in \cite[Lemma 5.1]{L2}, we can define plurisubharmonic weight functions
\[ \psi_i = C(d) \phi_i +\sum_{k=1,2} \lambda\Big(d^{-2} \sum_{j=1}^n |z^i_{kj}|^2\Big) \log \left( \sum_{j=1}^n |z^i_{kj}|^2\right) \]
for a suitable constant $C(d)$, and cutoff function $\lambda$ supported in $[0,1/2)$, equal to 1 in $[0,1/4]$. These have the property that $e^{-\psi_i}$ is not locally integrable at $x_1^i,x_2^i$.

Let us define the functions
\[ f_i = \lambda\Big( d^{-2} \sum_{j=1}^n |z^i_{1j}|^2\Big), \]
which equal 1 in $B(x^i_1, d/2)$, and zero outside of $B(x^i_1, d)$. We have a uniform bound $\Vert \bar\partial f_i\Vert_{L^2(e^{-\psi_i})} < C$ using the weight functions $e^{-\psi_i}$, and so applying the H\"ormander estimate (as in the proof of Theorem\ref{thm:chart} the ball $B(p_i, 6)$ is contained in a Stein domain), we can solve the equations $\bar\partial h_i = \bar\partial f_i$ on $B(p_i,6)$, with estimates $\Vert h_i\Vert_{L^2(e^{-\psi_i})} < C$. Note that near $x^i_1, x^i_2$ this implies that $\bar\partial h_i=0$, and since $e^{-\psi_i}$ is not locally integrable near these points, we have $h_i(x^i_k)=0$ for $k=1,2$. This shows that the holomorphic functions $f_i - h_i$ separate the points $x^i_1, x^i_2$.
Since the construction is uniform in $i$, we obtain holomorphic functions of the $z_i$ which separate $x_1, x_2$, and so $x_1,x_2$ must be separated by the $z_i$.
If we now let $F = (z_1,\ldots, z_n)$, and $\Omega = F^{-1}(B(0,4))$, we find that $F$ is proper and injective, and therefore it is a homeomorphism.
\end{proof}

Note that these local charts define a holomorphic atlas on $\mathcal{R}_\epsilon$. For this it is enough to note that if $f_i = f_i(z^i_1,\ldots, z^i_n)$ are uniformly bounded holomorphic functions, and $f_i \to f$ under the Gromov-Hausdorff convergence, then $f$ is a holomorphic function of $z_1,\ldots, z_n$.

To conclude this section, let us present two applications of Theorem \ref{thm:chart}.

\begin{prop}\label{prop:scalar}
Let $(M^n_i,\omega_i, p_i)$ be a sequence of complete K\"ahler manifolds with $\mathrm{Ric} > -1$ and $\mathrm{vol}(B(p_i, 1)) > v>0$. Assume that $(M^n_i, p_i)\to (M^n, p)$ in the pointed Gromov-Hausdorff sense, where $M^n$ is a smooth Riemannian manifold. Then the scalar curvatures $S_i$ of $M_i$ converge to the scalar curvature $S$ of $M$ in the measure sense. That is to say, for any points $M_i\ni q_i\to q\in M$, and any $r>0$, we have $\int_{B(q_i, r)}S_i \omega_i^n \to \int_{B(q, r)}S \omega^n$ as $i\to\infty$.
\end{prop}
\begin{remark}
It is clear from the proof that this proposition is local in nature. That is to say, the completeness of the K\"ahler metrics is not necessary, as long as $B(q_i, r)$ is relatively compact in $M_i$.
\end{remark}
\begin{proof}
It suffices to prove that there exists a subsequence of $M_i$ so that the proposition is true,
and we only need to prove the result locally near $p$. By suitable scaling, we may assume that $d_{GH}(B(p, \frac{1}{\epsilon^2}), B_{\mathbb{C}^n}(0, \frac{1}{\epsilon^2}))<\epsilon^2$ and $\mathrm{Ric}(M_i)\geq -\epsilon^2$. Here $\epsilon=\epsilon(n)$ is the small constant in Theorem \ref{thm:chart}.
Thus we have holomorphic charts $(z^i_1, ..., z^i_n)$ on $B(p_i, 10)$ so that the coordinate maps $(z^i_1, .., z^i_n): B(p_i, 10)\to\mathbb{C}^n$ give $\Psi(\epsilon|n)$-Gromov-Hausdorff approximations to their images. Let us assume that the holomorphic charts $(z^i_1, ..., z^i_n)$ converge to a chart $(z_1, ..., z_n)$ on $B(p, 8)$ as in Proposition~\ref{prop:limitchart}. This defines a complex structure $J$ on the ball $B(p,8)$. Note that a priori $g$ is just a Riemannian metric on $M$, however we have the following.

\begin{claim}
The metric $g$ on $M$ is compatible with the complex structure $J$ on $B(p, 8)$. That is to say, $g$ is a K\"ahler metric with respect to $J$.
\end{claim}
\begin{proof}
Note that the functions $z^i_j$ are all holomorphic, hence harmonic, and so $z_1, .., z_n$ are all complex harmonic on $B(p, 8)$. Therefore, they are all smooth with respect to the Riemannian metric $g$. In particular, this shows that the complex structure $J$ is smooth with respect to $g$. In addition, it follows that the composition of chart maps $(z_j)^{-1}\circ (z^i_j)$ gives a holomorphic $\Psi(i^{-1})$-Gromov-Hausdorff approximation from $B(p_i,7)$ to its image in $B(p,8)$.

Let $u_i$ be such that $\sqrt{-1}\partial\overline\partial u_i =\omega_i$ on $B(p_i, 10)$, given by Theorem~\ref{thm:chart}. As $\Delta u_i = 2n$, $\nabla u_i$ is uniformly bounded on $B(p_i, 9.5)$, and we can assume that $u_i$ converges uniformly under the Gromov-Hausdorff convergence to a smooth function $u$ on $B(p, 9)$, satisfying $\Delta u = 2n$. Using our charts, the $u_i$ can be viewed as plurisubharmonic functions on $B(p,9)$ converging uniformly to $u$, and so $\omega=\sqrt{-1}\partial\overline\partial u$ is a closed positive $(1, 1)$ current with smooth coefficients. To prove the claim, it suffices to prove that the metric $g$ is the same as the K\"ahler metric $\omega$. Note that $\omega$ is a well defined form on $M$, independent of the choice of bounded K\"ahler potentials $u_i$ for $\omega_i$. This follows since if above we choose different potentials $u_i'$ converging to $u'$, then $v_i = u_i - u_i'$ are pluriharmonic, and so is their uniform limit. So $\sqrt{-1}\partial\bar\partial u' = \sqrt{-1}\partial\bar\partial u$.

To compare these two smooth metrics, we can blow up a point $p$ on $M$ and compare the metrics on the tangent space.
Let $\tau > 0$ be a small number. Let us rescale the distance on $M_i$ and $M$ by $\frac{1}{\tau}$. Let the rescaled manifolds be $(M^\tau_i, p_{\tau, i})$ and $(M^\tau, p_{\tau})$, and the metrics be $\omega_{\tau, i}$ and $\omega_\tau$, $g_\tau$.
Then by applying Theorem~\ref{thm:chart} again, on $B(p_{\tau, i}, 1)$ we have $\omega_{\tau, i} = \sqrt{-1}\partial\overline\partial u_{\tau, i}$. Let us say $u_{\tau, i}\to u_\tau$ on $B(p_\tau, 1)$. Notice that the limit of $g_\tau$ as $\tau\to 0$ is the Euclidean metric, and by Theorem~\ref{thm:chart} the limit potentials $u_\tau$ approach the distance squared $|z|^2$ from the origin as $\tau\to 0$. It follows that the limit of $\sqrt{-1}\partial\overline\partial u_\tau$ as $\tau\to 0$ is the K\"ahler form associated to the Euclidean metric. This implies our claim.
\end{proof}

From now on, for a function $u$ on $B(p, 7)$, we may also think it is defined on $B(p_i, 6)$ by lifting via the coordinate map.

\begin{claim}\label{cl1}
$\int_{B(p_i, 5)}|\langle dz^i_{k}, d\overline{z^i_j}\rangle -\langle dz_k, d\overline{z_j}\rangle|^2<\Psi(\frac{1}{i})$. \end{claim}
\begin{proof}
 Pick $q\in B(p, 6)$ and $M_i\ni q_i\to q$.
From a standard covering argument, it suffices to prove that \begin{equation}\label{-1}\lim\limits_{\rho\to 0}\lim\limits_{i\to\infty}\dashint_{B(q_i, \rho)}|\langle dz^i_{k}, d\overline{z^i_{j}}\rangle -\langle dz_k, d\overline{z_j}\rangle|^2=0.\end{equation}
By subtracting constants, we may assume $z^i_{j}(q_i) = z_j(q) = 0$ for all $i, j$, and in addition, applying a linear transformation we can assume that $\langle dz_k, d\bar z_j\rangle (q) = \delta_{kj}$.
Let us rescale the distance on $(M_i, q_i)$ and $(M, q)$ by $\frac{1}{\rho}$. We also rescale $z^i_{j}, z_j$ by $\frac{1}{\rho}$. As $M$ is a smooth manifold, $\langle dz_k, d\overline{z_j}\rangle$ is a smooth function and so after scaling we have $\langle dz_k, d\overline{z_j}\rangle = \delta_{jk} + \Psi(\rho)$ on $B(q,1)$. It follows that for sufficiently large $i$, the $z^i_j$ define a $\Psi(\rho)$-Gromov-Hausdorff approximation to the Euclidean ball, and so by \cite{CC}, we have
\begin{equation}\label{eq:CCestimate}
  \dashint_{B(q_i,1)} | \langle dz^i_k, d\bar z^i_j\rangle - \delta_{jk}|^2 < \Psi(\rho).
\end{equation}
The required equality \eqref{-1} follows from this.
\end{proof}

Set $s_i = dz^i_{1}\wedge dz^i_{2}\wedge\ldots\wedge dz^i_{n}$ and $s = dz_{1}\wedge dz_{2}\wedge\ldots\wedge dz_{n}$. Then Claim \ref{cl1} implies that \begin{equation}\label{-2}\int_{B(p_i, 5)}\left||s|^2-|s_i|^2\right|<\Psi(\frac{1}{i}).\end{equation}
\begin{lemma}\label{lm1}
$\lim\limits_{i\to\infty}\int_{B(p_i, 4)}\left|\log |s_i|^2-\log |s|^2\right|\omega_i^n = 0$.
\end{lemma}
\begin{proof}
The Poincar\'e-Lelong equation says \begin{equation}\label{-237}\frac{\sqrt{-1}}{2\pi}\partial\overline\partial\log|s_i|^2 =\mathrm{Ric}(M_i)\geq -\epsilon\omega_i.\end{equation}
According to Theorem \ref{thm:chart}, we may assume that on $B(p_i, 10)$, $\omega_i = \sqrt{-1}\partial\overline\partial u_i$ and $u_i\to u$ uniformly. Let us say $|u_i|\leq 1000$ for all $i$.
By \eqref{-237}, $\log|s_i|^2+2\pi\epsilon u_i$ is plurisubharmonic on $B(p_i, 10)$. Set
\begin{equation}\label{eq:vi}
  \begin{aligned}
    v_i&=\log |s_i|^2+2\pi\epsilon u_i, \\
    v&=\log |s|^2+2\pi\epsilon u.
  \end{aligned}
\end{equation}
We need to show that $v_i$ converges to $v$ in $L^1$.
If we were working on the same space, then the lemma would follow from the standard theory of plurisubharmonic functions since by (\ref{-2}), $\log |s_i|^2$ cannot go uniformly to $-\infty$.

\begin{claim}\label{cl2}
There exists a constant $C$, independent of $i$, so that $\int_{B(p_i, 9)}|v_i| \leq C$.
\end{claim}
\begin{proof}The argument will be similar to Proposition $2.7$ of \cite{L4}. In the proof $C$ will be a large constant independent of $i$. The value might change from line to line.
Since $v_i$ has an upper bound independent of $i$, it suffices to prove $\int_{B(p_i, 7)}v_i$ is uniformly bounded from below.
Let $G_i(x, y)$ be the Green's functions on $B(p_i, 10)$. Set \begin{equation}\label{-3}F_i(x) = v_i(x) + \int_{B(p_i, 10)}G_i(x, y)\Delta v_i(y)dy.\end{equation} Then $F_i$ is harmonic, with the same boundary values as $v_i$. By the maximum principle, $F_i\leq \sup\limits_{B(p_i, 10)}v_i$. Let us say $F_i\leq C$. From (\ref{-2}), we can find a point $x_i\in B(p_i, 1)$ so that $v(x_i)\geq -C$ for all $i$. Since $v_i$ is subharmonic, $F_i(x_i)\geq -C$. Then by the Cheng-Yau gradient estimate \cite{CY}, we have $|F_i|\leq C$ on $B(p_i, 9)$. Inserting $x = x_i$ in (\ref{-3}), we also obtain that $\int_{B(p_i, 9)}\Delta v_i(y) dy\leq C$, using the lower bound for the Green's function. By changing the radius $10$ to $11$ in (\ref{-3}), we may assume that $\int_{B(p_i, 10)}\Delta v_i(y) dy\leq C$.

By integrating (\ref{-3}), we find
\begin{equation}\label{-4}\int_{B(p_i, 9)}v_i(x) = \int_{B(p_i, 9)}F_i(x) - \int_{B(p_i, 10)}\left(\int_{B(p_i, 9)}G_i(x, y)dx\right)\Delta v_i(y)dy\geq -C,
\end{equation}
using the upper bound for the Green's function.
\end{proof}

To complete the proof of Lemma~\ref{lm1} we will use the argument in H\"ormander~\cite{Ho}, Theorem $4.1.9$ on page $94$.
Because of Claim \ref{cl2}, by passing to a subsequence, we may assume that $v_i$ converges weakly as a measure to $w$ on $B(p, 7)$.
By this we mean that for any smooth function $u$ with compact support on $B(p, 9)$, $\int_{M_i}uv_i\to \int_M uw$. We emphasize that a priori, $w$ is merely a measure.
In H\"ormander's proof, the convolution is used to mollify the functions. Since $M_i$ is not necessarily Euclidean, we consider the heat flow.
More precisely, let $\phi$ be a smooth cut-off function on $M$ so that $\phi = 1$ on $B(p, 8)$ and $\phi = 0$ outside $B(p, 9)$. Recall that we use the charts to identify the different balls $B(p_i, 9)$, and by the gradient estimate we have a uniform lower bound for $\omega_i$ in the charts. It follows that we can assume $|\phi|, |\nabla\phi|, |\Delta\phi|<C$ with respect to the metric $\omega$ on $M$ as well as with respect to the metrics $\omega_i$.
Define
\begin{equation}
  \begin{aligned} \label{-5}v_{it}(x) &= \int_{M_i} H_i(x, y, t)\phi(y)v_i(y)dy, \\
    w_t(x) &= \int_M H(x, y, t)\phi(y)w(y)dy,
  \end{aligned}
\end{equation}
where $H_i, H$ are heat kernels on $M_i$ and $M$ respectively. Note that for any $t>0$, $w_t$ is a function.
 \begin{equation}\begin{aligned}\frac{dv_{it}}{dt} = &\int_{M_i}H_i(x, y, t)v_i(y)\Delta\phi(y)dy+ \int_{M_i}H_i(x, y, t)\phi(y)\Delta v_i(y)dy\\&+ 2\int_{M_i}H_i(x, y, t)\langle\nabla v_i(y), \nabla\phi(y)\rangle dy.\end{aligned}\end{equation}
Notice that $\Delta\phi$ has support outside of $B(p_i, 7)$.  As $v_i(y)$ has a uniform $L^1$ bound, according to Li-Yau's heat kernel estimate \cite{LY}, on $B(p_i, 6)$, the first term will be bounded by $\Psi(t)$. The second term is nonnegative, since $v_i$ is subharmonic. For the last term, we can do integration by parts to transform derivatives to the heat kernel and $\phi$. By the estimate for the derivative of the heat kernel \cite{Ko}, we find that on $B(p_i, 6)$, the last term is bounded by $\Psi(t)$. Therefore, we find that on $B(p_i, 6)$, \begin{equation}\label{-6}\frac{dv_{it}}{dt}\geq -\Psi(t).\end{equation}

Let $\eta$ be a nonnegative smooth function so that $\eta = 1$ on $B(p, 4)$, $\eta = 0$ outside $B(p, 5)$. According to the assumption,
\begin{equation}\label{-7}\int_{M_i}v_i\eta\to \int_Mw\eta.\end{equation}
It is clear from (\ref{-5}) that
\begin{equation}\label{-8}\left|\int_M(w\eta-w_t\eta)\right|\leq \Psi(t).\end{equation}
For each fixed $t>0$, we have $H_i(x, y, t)\to H(x, y, t)$. Therefore, $v_{it}\to w_t$ uniformly on each compact set. By the almost monotonicity (\ref{-6}), given any $\delta>0$, there exists $a>0$ sufficiently small so that on $B(p_i, 6)$, \begin{equation}\label{-9}w_a-v_i+\delta>0\end{equation} for all sufficiently large $i$. Note that by the volume convergence theorem of Colding \cite{C}, we have \begin{equation}\label{-10}\int_{M_i}w_a\eta \to \int_Mw_a\eta.\end{equation}
Putting (\ref{-7})-(\ref{-10}) together, we find that for sufficiently large $i$, \begin{equation}\label{eq:b1}
  \int_{M_i}|w_a-v_i+\delta|\eta\leq \Psi(a, \delta).
\end{equation}
Then we obtain that \begin{equation}\label{-11}\lim\limits_{\lambda\to+\infty}\liminf\limits_{i\to\infty}\int_{E_i^\lambda}v_i = 0,\end{equation} where $E_i^\lambda = \{x\in B(p_i, 4)|v_i(x)\leq -\lambda\}$.  Notice that on $B(p, 4)$, $\log |s|^2$ is bounded from below. Recall that $v_i, v$ are defined in \eqref{eq:vi}. Lemma \ref{lm1} follows by putting (\ref{-2}) and (\ref{-11}) together.
\end{proof}

Let $h$ be a smooth function of compact support on $B(p, 4)$. Note that by the Poincar\'e-Lelong equation, the scalar curvature is given by
\begin{equation}\label{-12}S_i = \frac{1}{2\pi}\Delta\log |s_i|^2.\end{equation} Therefore
\begin{equation}\label{-13}\begin{aligned}2\pi \left(\int_{M_i}hS_i -\int_MhS\right) &= \int_{M_i}\log |s_i|^2\Delta h - \int_M\log |s|^2\Delta h\\&=\int_{M_i}(\log |s_i|^2-\log |s|^2)\Delta h+\int_{M_i}\log |s|^2(\Delta_{\omega_i}h-\Delta_\omega h)\\&+(\int_{M_i}\log |s|^2\Delta_\omega h-\int_M\log |s|^2\Delta h).\end{aligned}\end{equation} As we noticed before, $\Delta h$ with respect to $\omega_i$ is uniformly bounded. Therefore, according to Lemma \ref{lm1}, the first term approaches zero as $i\to\infty$. Note that by Claim \ref{cl1},  $\int_{M_i}|\Delta_{\omega}h-\Delta_{\omega_i}h|\to 0$. Therefore, the second term converges to zero. Finally, the last term converges to zero by the volume convergence theorem of Colding \cite{C}. We obtained the following.
\begin{lemma}\label{lm2}
Let $h$ be a smooth function of compact support on $B(p, 4)$. Then $\lim\limits_{i\to\infty}\int_{M_i}hS_i = \int_MhS$.
\end{lemma}

Note that $S_i + 2n\geq 0$ for all $i$.
For any $r_1<r<r_2$, we can find smooth functions $f$ and $g$ so that $0\leq f, g\leq 1$; $f = 1$ on $B(p, r_1)$, $f$ has compact support on $B(p, r)$; $g=1$ on $B(p, r)$, $g$ has compact support on $B(p, r_2)$. For sufficiently large $i$ we have \begin{equation}\label{-14}\int_{M_i}f(S_i+2n)\leq \int_{B(p_i, r)}(S_i+2n)\leq \int_{M_i}g(S_i+2n).\end{equation}
We can apply Lemma \ref{lm2}. Letting $i\to\infty$ and $r_1, r_2\to r$, we obtain the proof of
Proposition \ref{prop:scalar}.
\end{proof}

Corollary~\ref{cor:Sintsmall} is immediate from Proposition~\ref{prop:scalar}. We will prove Proposition~\ref{prop:Sintbdd} after Proposition~\ref{prop:cone2}.

Finally we prove Proposition~\ref{prop:Cn}. We state it again here for convenience.
\begin{prop}\label{prop:gapthm}
There exists $\epsilon(n)>0$ so that if $M^n$ is a complete noncompact K\"ahler manifold with $\mathrm{Ric}\geq 0$ and $\lim\limits_{r\to\infty} r^{-2n}\mathrm{vol}(B(p, r))\geq \omega_{2n}-\epsilon$, then $M$ is biholomophic to $\mathbb{C}^n$.
\end{prop}
The argument is very similar to \cite[Theorem 4.1]{L2}, with the three circle theorem replaced by the three annulus type result from Donaldson-Sun~\cite[Proposition 3.7]{DS2}. In our setting, the relevant statement is the following.
\begin{lemma}\label{lem:3annulus}
  Given $\delta\in (0,1)$ there exists $\epsilon > 0$ with the following property. Suppose that $\lim\limits_{r\to\infty} r^{-2n}\mathrm{vol}(B(p, r))\geq \omega_{2n}-\epsilon$. Then for any $r > 0$, and any holomorphic function $f$ on $B(p,2r)$ we have that
  \[
      \dashint_{B(p,r)} |f|^2 \geq 2^{2(1-\delta)} \dashint_{B(p,r/2)} |f|^2 \,\text{ implies }\,
      \dashint_{B(p,2r)} |f|^2 \geq 2^{2(1-\delta)} \dashint_{B(p,r)} |f|^2,
    \]
    and
    \[
      \dashint_{B(p,r)} |f|^2 \geq 2^{2(1+\delta)} \dashint_{B(p,r/2)} |f|^2 \,\text{ implies }\,
      \dashint_{B(p,2r)} |f|^2 \geq 2^{2(1+\delta)} \dashint_{B(p,r)} |f|^2.
     \]
   \end{lemma}
   \begin{proof}
     By the volume monotonicity we have $r^{-2n}\mathrm{vol}(B(p,r))\geq \omega_{2n} - \epsilon$ for all $r$. After rescaling, we can assume that in the statement of the Lemma we have $r=1$. We can then argue by contradiction, just as in \cite[Proposition 3.7]{DS2} (see also Ding~\cite[Theorem 0.7]{Ding}). If the first conclusion were to fail, then we could extract a limiting harmonic function $f$ on the Euclidean ball $B_{\mathbb{C}^n}(0,2)$, such that
   \[ \dashint_{B(0,1)} |f|^2 \geq 2^{2(1-\delta)} \dashint_{B(0,1/2)} |f|^2, \,\text{ but }\,
     \dashint_{B(0,2)} |f|^2 \leq 2^{2(1-\delta)} \dashint_{B(0,1)} |f|^2. \]
   Such an $f$ would have to be homogeneous of degree $1-\delta$, but there is no such harmonic function on Euclidean space. The other conclusion follows similarly.
 \end{proof}

 \begin{proof}[Proof of Proposition~\ref{prop:gapthm}]
   First note that once $\epsilon$ is sufficiently small, we can apply Theorem~\ref{thm:chart} to arbitrary balls in $M$. In particular for any $R$ we obtain holomorphic functions $z^R_1,\ldots, z^R_n$ which provide a $\Psi(\epsilon)$-Gromov-Hausdorff approximation from $B(p, 2R)$ to $B_{\mathbb{C}^n}(0,2R)$. We can assume $z^R_i(p)=0$.  Once $\epsilon$ is sufficiently small, we have
   \[ \dashint_{B(p, 2R)} |z^R_i|^2 \leq 2^3 \dashint_{B(p,R)} |z^R_i|^2, \]
   and so iterating Lemma~\ref{lem:3annulus} we have
   \[ \dashint_{B(p,2^k)} |z^R_i|^2 \leq C2^{3k} \dashint_{B(p,1/2)} |z^R_i|^2, \]
   whenever $2^k < R$. Let us define $u^R_i$ in the span of the $z^R_i$, so that
   \[ \dashint_{B(p,1)} u^R_i \bar u^R_j = \delta_{ij}. \]
   The estimate above implies that we can extract limit holomorphic functions $u_1,\ldots, u_n$ on $M$ as $R\to\infty$, that are $L^2$-orthonormal on $B(p,1)$, and
   \[ \dashint_{B(p, R)} |u_i|^2 \leq C R^{3/2} \]
   for all $R > 1$.
We will show that $(u_1,\ldots, u_n)$ provides a biholomorphism from $M$ to $\mathbb{C}^n$ if $\epsilon$ is sufficiently small.

We next prove the properness of the map given by the $u_i$. For $R > 1$ let us choose a new basis $v^R_i$ for the span of the $u_i$, so that
\[ \dashint_{B(p,R)} v^R_i \overline{v^R_j} = c^R_i \delta_{ij}, \text{ and } \dashint_{B(p,1)} v^R_i \overline{v^R_j} = \delta_{ij}, \]
for some constants $c^R_i$. We then have $\sum |v^R_i|^2 = \sum |u_i|^2$. Let $\lambda^R_i = \sup_{B(p, R)} |v^R_i|$, and define $w^R_i = v^R_i / \lambda^R_i$. Then just as in \cite[Claim 4.2]{L2}, an argument by contradiction shows that if $\epsilon$ is sufficiently small, then the $Rw^R_i$ give an $\frac{R}{100n}$-Gromov-Hausdorff approximation from $B(p,R)$ to the Euclidean ball $B_{\mathbb{C}^n}(0,R)$. In particular $\sum |w^R_i|^2 > 1/2$ on $\partial B(p,R)$. At the same time, using Lemma~\ref{lem:3annulus} with $\delta=1/2$, and the fact that $v^R_i(p)=0$ (so that $v^R_i$ has at least linear growth on small scales), we have
\[ \dashint_{B(p, R)} |v^R_i|^2 \geq C^{-1} R \dashint_{B(p,1)} |v^R_i|^2 = C^{-1}R. \]
It follows that $\lambda^R_i \geq C^{-1/2}R^{1/2}$. Therefore on $\partial B(p,R)$ we have
\[ \sum |u_i|^2 = \sum |v_i^R|^2 \geq C^{-1/2}R^{1/2} \sum |w_i^R|^2 \geq \frac{1}{2} C^{-1/2}R^{1/2}. \]
This implies that the map given by $(u_1,\ldots, u_n) : M \to \mathbb{C}^n$ is proper.

We assert that the holomorphic $n$-form $du_1\wedge \ldots \wedge du_n$ cannot vanish at any point. Note that on each ball $B(p,R)$ the functions $u_i$ are obtained as a limit of $u_i^R$ which satisfy that $du_1^R\wedge\ldots \wedge du_n^R$ is nowhere vanishing. Therefore if $du_1\wedge \ldots \wedge du_n$ were to vanish at a point, then it would have to be identically zero. In this case the image of $(u_1,\ldots, u_n)$ would have dimension at most $n-1$ in $\mathbb{C}^n$, so the preimage of a point would be a compact subvariety in $M$ of dimension at least $1$. This contradicts that by Theorem~\ref{thm:chart} we can find holomorphic charts on arbitrarily large balls in $M$. Since $\mathbb{C}^n$ is simply connected, it follows that $(u_1,\ldots, u_n)$ must be a biholomorphism.
\end{proof}

\section{Constructing projective embeddings}\label{sec:partialc0}
Suppose that
$(M_i^n, \omega_i, L_i)$ is a sequence of compact polarized K\"ahler manifolds with $\mathrm{Ric}(M_i) > -1, \mathrm{diam}(M_i) < d, \mathrm{vol}(M_i) > v$ and such that the curvature of $L_i$ is the K\"ahler metric $\omega_i$. Let us assume that the sequence $M_i$ converges in the Gromov-Hausdorff sense to the limit $X$. Our goal is to show that $X$ is homeomorphic to a normal projective variety, following Donaldson-Sun~\cite{DS1}.
As in \cite{DS1}, the main step is the construction of holomorphic sections of suitable powers of $L_i$, uniformly in $i$. We can follow the argument in \cite{DS1} fairly closely, the main new difficulty being that in our setting we do not have smooth convergence of the metrics on the regular set. To overcome this we will use the existence of good holomorphic charts on the regular set (or rather the set $\mathcal{R}_\epsilon$ for small $\epsilon$) provided by Theorem~\ref{thm:chart}. The argument is simplified by using the recent estimate of Cheeger-Jiang-Naber~\cite{CJN} on the codimension 2 Minkowski content of the singular set, but see the Appendix for a proof which avoids this.

We prove the following.
\begin{prop}\label{prop:DS1}
  Given $\nu, \zeta > 0$ there are $K, \epsilon, C > 0$ with the following property. Let $(M^n,L,\omega)$ be a polarized K\"ahler manifold such that $\mathrm{Ric}(\omega) > -\epsilon\omega$, and $\mathrm{vol}(B(q,1)) > \nu$ for all $q\in M$. Suppose that $d_{GH}(B(p,\epsilon^{-1}), B(o, \epsilon^{-1})) < \epsilon$ for a metric cone $(V,o)$. Then $L^m$ admits a holomorphic section $s$ over $M$ for some $m < K$, such that $\Vert s\Vert_{L^2} \leq C$,
and  $| |s(x)| - e^{-md(x,p)^2/2}| < \zeta$ for $x\in B(p, 1)$.
\end{prop}
Note that if $(M^n,  \omega)$ satisfies $\mathrm{Ric} > -1, \mathrm{diam}(M) < d, \mathrm{vol}(M) > v$, then by volume comparison, there is a $\nu > 0$, depending on $d, v$ such that $\mathrm{vol}(B(p,r)) > \nu r^{2n}$ for all $r \leq 1$. Thus scaled up copies of $M$ will satisfy the local non-collapsing assumption in the statement of the proposition.

\begin{proof}
  We will argue by contradiction, so suppose that the sequence $(M_i^n, p_i, L_i, \omega_i)$ satisfying the assumptions converges in the pointed Gromov-Hausdorff sense to $(V, o)$. We will show that there is an $m > 0$ such that for sufficiently large $i$,  $L_i^m$ admits a suitable section for the points $p_i$.

From Theorem~\ref{thm:chart}, volume comparison, and Cheeger-Colding theory, we have an $\epsilon = \epsilon(n) > 0$ such that if $q\in V$ satisfies $r^{-2n}\mathrm{vol}(B(q,r)) > \omega_{2n} - \epsilon$ for some $r < \epsilon$, then we have a holomorphic chart on $B(q,\epsilon r)$. In terms of this we define the singular set $\Sigma$ of $V$ as points $q$ which satisfy $\lim\limits_{r\to 0}r^{-2n} \mathrm{vol}(B(q, r))\leq\omega_{2n}-\epsilon$. Note that $\Sigma$ is a closed set.
Let $\Sigma_\rho$ be the $\rho$-neighborhood of $\Sigma$ and set $U = B(o, R)\backslash \Sigma_\rho$. Here $R$ is sufficiently large and $\rho$ if sufficiently small, to be chosen later, depending on the parameters $n, \nu$.

We can cover $U$ by a finite number of small geodesic balls $U'^j$ ($j = 1, .., N$, $N$ depends on $R, \rho$) with center $q_j$, such that  the balls $U^j=\epsilon U'^j$ with the same centers, but radius scaled by $\epsilon$, still cover $U$. The radii of the $U'^j$ can be chosen to be a fixed small constant, depending on $R,\rho$.

Let $U_i$ be the lift of $U$ back to $M_i$, under the Gromov-Hausdorff approximation. By Theorem~\ref{thm:chart} and our choice of $\epsilon$, for sufficiently large $i$, we have uniform holomorphic charts $(U^j_i, (z^j_{i1}, ..., z^j_{in}))$ covering $U_i$. Moreover, by Proposition~\ref{prop:limitchart} the holomorphic charts on $U^j_i$ converge to charts on $U^j$, with holomorphic transition functions. Thus $U$ admits a holomorphic structure.

From Theorem~\ref{thm:chart}, on each $U^j_i$, there is a function $\rho_{ij}$ so that $\sqrt{-1}\partial\overline\partial\rho_{ij} = \omega_i$ and $\rho_{ij}$ is very close to the distance squared from the center of $U^j_i$.
In particular these $\rho_{ij}$ are uniformly bounded, independent of $i$. Observe $\Delta \rho_{ij} = 2n$, so Yau's estimate gives us uniform gradient bound on a smaller interior domain. Therefore, by passing to a subsequence, we may assume $\rho_{ij}\to \rho_j$ uniformly.

\begin{claim} \label{claim:pluriharmonic}
On each $U^j$, $\rho_j-r^2$ is a pluriharmonic function with respect to the holomorphic structure constructed above, where $r$ is the distance function to the vertex $o \in V$.
\end{claim}
\begin{proof}
First, $\rho_j-r^2$ is a bounded Lipschitz function. To prove the result, we show that if $\alpha$ is any smooth $(n-1,n-1)$ form with respect to the holomorphic structure, with compact support in $U^j$, then
\[ \int_{U^j}(\rho_j-r^2)\partial\overline\partial \alpha = \int_{U^j} \partial\bar\partial (\rho_j - r^2) \wedge \alpha = 0. \]
According to Cheeger-Colding \cite{CC}, there is a sequence of functions $h_i$ on the $U_i^j$ converging to $r^2$ as $i\to\infty$, and such that
\[ \lim_{i\to\infty} \int_{U_i^j} |\omega_i - \sqrt{-1}\partial\bar\partial h_i|_{\omega_i}^2\, \omega_i^n = 0. \]
The claim then follows from the fact that under our charts
\[\lim\limits_{i\to\infty}\left|\int_{U^j_i}\partial\overline\partial (\rho_{ij}-h_i)\wedge \alpha\right|\leq \lim\limits_{i\to\infty}\int_{U^j_i}|\partial\overline\partial (\rho_{ij}-h_i)|_{\omega_i}|\alpha|_{\omega_i}\omega_i^n = 0.\]
Here we used that $|\alpha|_{\omega_i}$ is uniformly bounded since our holomorphic charts have uniformly bounded gradients, and so the $\omega_i$ have uniform lower bounds in terms of the charts.
\end{proof}

Since $U^j_i$ is contained in a holomorphic chart, the line bundle $L_i$ over $U^j_i$ is isometric to a trivial holomorphic line bundle with weight $e^{-\rho_{ij}}$.
Let $s_{ij}$ be a holomorphic section over $U^j_i$ so that $|s_{ij}| = e^{-\rho_{ij}}\neq 0$. Note that if $U_i^j\cap U^k_i\neq \emptyset$, the (holomorphic) transition functions $f_{ijk}=\frac{s_{ik}}{s_{ij}}$ are uniformly bounded. Therefore, after taking a subsequence, we may assume that the line bundles $L_i$ converge to a Hermitian holomorphic line bundle $(L,h)$ over $U$.

The line bundle $L$ is trivial over each $U_j$ with weight $e^{-\rho_j}$.
By Claim~\ref{claim:pluriharmonic} the metric $e^{r^2}h$ on $L$ is flat over $U$, but since $U$ is not necessarily simply connected, $(L,e^{r^2}h)$ need not be a trivial holomorphic line bundle with the flat metric. To deal with the possible presence of holonomy we follow Donaldson-Sun's argument. For the reader's convenience, we include some details. Pick a point $q\in U\cap \partial B_V(o, 1)$. As $U$ is connected, we can join $q$ with $q_j$ (recall $q_j$ is the center of geodesic balls $U_j$) by smooth curves $l_j\subset U$ (in terms of holomorphic structure of $U$). Now let $s'$ be a vector in the fibre of $L$ over $q$ so that $|s'| = 1$. We can parallel transport $s'$ along $l_j$, to get vectors $s'_j$ in the fibre of $L$ over $q_j$. Let us parallel transport $s'_j$ in the geodesic ball $U_j$. Since $L$ is flat and $U_j$ is contained in a holomorphic chart, $s'_j$ is well-defined.

\begin{claim}Given any $\delta>0$,
if we replace $L$ by $L^m$, where $m$ is some number that is bounded by a constant $K(\delta, n ,\nu, R, \rho)$, then in the overlap $U_j\cap U_k$ we can ensure that $|s'_j-s'_k|_{e^{r^2}h}<\delta$.
\end{claim}
\begin{proof}
Notice that the norms under $e^{r^2}h$ of the $s'_j$ are all equal to one. So $s'_j = s'_k e^{\sqrt{-1}\theta_{jk}}$. As $U_j\cap U_k$ is connected, $\theta_{jk}$ is constant.
$(s'_j)^m = (s'_k)^m e^{\sqrt{-1}m\theta_{jk}}$. To prove the claim, we just need to find $m$ so that $m\theta_{jk}$ is close to $2\pi$ times an integer for all $j, k\leq N$.
It follows from elementary number theory that such $m$ exists. The bound follows from the fact that $N$ depends only on $n, \nu, R, \rho$.\end{proof}

Let us fix a very small $\delta = \delta(n, \nu)$, to be determined later.
From now on, we replace $L$ by $L^m$, where $m < K(\delta, n,\nu, R, \rho)$. Let us still call the new line bundle $L$. Let us also rescale the metric $(M_i, p_i, \omega_i, L_i)$ by $(M_i, p_i, m\omega_i, L_i^m)$. Since $m$ is a fixed number, the new sequence which we still call $(M_i, p_i, \omega_i)$ will converge to $(V, o)$.  We can consider the same $U\subset V$.  A priori, the charts $U^j_i$ might be different, but we shall make the centers $q_j$ be the same. Note that $|s_j'-s_k'|_{e^{r^2}h}$ reflects the holonomy and homotopy preserves the holonomy since $L$ is flat. Thus the claim implies that $|s_j'-s_k'|_{e^{r^2}h}<\delta$. Note that $s_j'$ is a holomorphic section of $L$ over $U^j$. Under the convergence $L_i\to L$, we can find holomorphic sections $s^i_j$ on $L_i$ over $U^j_i$ so that $s^i_j\to s'_j$. Then, for sufficiently large $i$  we have $|s^i_j-s^i_k|\leq 2\delta e^{-\frac{1}{2}r^2}\leq 2\delta$. Moreover, $|s^i_j|\leq 10e^{-\frac{1}{2}r^2}$ for large $i$.

By a standard partition of unity, we can glue the sections $s^i_j$ together to a smooth section $\hat{s}_i$ of $L_i$ over $U_i$ so that
$|\overline\partial\hat{s}_i|^2<\gamma$, $|\hat{s}_i-s^i_j|<\min(10\delta, 20e^{-\frac{1}{2}r^2})$. Here $\gamma$ is a small number depending only on $n, \nu, \delta, R, \rho$, and we are using the smooth structure given by our holomorphic charts.

Similarly to \cite{DS1}, we introduce the first standard cut-off function $\psi^1_i$, supported in $B(p_i, R)$, and the second cut-off function $\psi^2_i$, supported outside $B(p_i, \rho)$.

To define the third cut-off function $\psi^3_i$, first recall the cut-off function in \cite{LT}, page $871$.
More precisely, if $\epsilon'\ll\epsilon\ll 1$ (constants independent of $i$) we define a cut-off function $\psi(t) = 1$, if $t\geq \epsilon$; $\psi(t) =(\frac{t}{\epsilon})^\epsilon$ if $2\epsilon'\leq t\leq \epsilon$; $\psi(t) =(2\frac{\epsilon'}{\epsilon})^\epsilon(\frac{t}{\epsilon'}-1)$, if $\epsilon'\leq t\leq 2\epsilon'$; $\psi(t) =0$ otherwise.

Let $\Sigma^i\subset M_i$ so that $\Sigma^i$ converges to the singular set $\Sigma$ under the Gromov-Hausdorff approximation. Let $\Sigma^i_r$ be the $r$-tubular neighborhood of $\Sigma^i$.
For $x\in B(p_i, R)$, let $d_i(x) = dist(x, \Sigma^i)$.
According to Cheeger-Jiang-Naber's theorem \cite{CJN} and the volume convergence theorem \cite{C}, for $\epsilon'<r<2\epsilon$, if $i$ is sufficiently large, $\mathrm{vol}(\Sigma^i_{r}\cap B(p_i, R))\leq C(n, \nu, R)r^2$.
Let the third cut-off function be $\psi^3_i(x) =\psi(d_i(x))$. Note $|\psi'|$ is decreasing from $2\epsilon'$ to $\epsilon$ and $10|\psi'(2t)|\geq |\psi'(t)|$ for any $2\epsilon'\leq t\leq \frac{1}{2}\epsilon$.
By the calculation in \cite{LT}, for $i$ large enough, we can make $\int_{B(p_i, R)}|\nabla\psi^3_i|^2e^{-r^2}$ as small as we want, provided $\epsilon'$ and $\epsilon$ are small enough.

Recall that $U$ is the complement of $\Sigma_\rho$ in  $B(o, R)$. Let us assume $\rho<\frac{1}{10}\epsilon'$. As in \cite{DS1}, the smooth section $\tilde{s_i}= \psi^1_i\psi^2_i\psi^3_i\hat{s}_i$ satisfies
\begin{itemize}
\item
 $\tilde{s}_i$ supported in $B(p_i, R)\backslash (B(p_i, \rho)\cup \Sigma^i_{\epsilon'})$.
\item $\int |\overline\partial \tilde{s}_i|^2<\gamma_2$ (can be as small as we want, if we set the parameters $R, \rho, \epsilon, \epsilon'$ properly).
\item  $|\nabla \tilde{s}_i|\leq C(n, \nu)$ on $U_i\backslash \Sigma^i_{10\epsilon}$
\item $|\tilde{s}_i-s^i_j|\leq C(n, \nu)\delta$ on a slightly smaller subdomain of $U_i$.
\end{itemize}
 By H\"ormander's $L^2$ estimate, we can find a holomorphic section $s_i$ on $M_i$ so that if we set $s''_i = s_i-\tilde{s}_i$, then $\int_{M_i}|s''_i|^2\leq 10\gamma_2$. As $s_i$ has uniform $L^2$ bound, $|\nabla s_i|$ is uniformly bounded. Thus, on $U_i\backslash \Sigma^i_{10\epsilon}$, $\nabla s''_i$ is uniformly bounded. Therefore, by the integral estimate of $s''_i$, $s''_i$ is very small in $V_i=U_i\backslash \Sigma^i_{\epsilon_0}$. Here $\epsilon_0$ is a small number depending only on $n, \nu, \gamma_2$. This means that on $V_i$, $s_i$ is close to $\tilde{s}_i$, hence $|s_i|^2$ is close to $e^{-r_i^2}$ on $V_i$ (here $r_i$ is the distance to $p_i$). But as a set, $V_i$ is Hausdorff close to $B(p_i, R)$. Then by the gradient estimate of $s_i$, we find that $|s_i|^2$ is very close to $e^{-r_i^2}$ on $B(p_i, R)$.

  Since $\tilde{s}_i$ vanishes outside of $B(p_i, R)$, $s''_i$ is holomorphic on $M_i\backslash B(p_i, R)$. As $\int |s''_i|^2\leq 10\gamma_2$, we can make sure that $|s''_i|^2<c(n, \nu, d)\gamma_2$ on $M_i\backslash B(p_i, 2R)$. Now we can choose appropriate parameters so that Proposition \ref{prop:DS1} holds.
\end{proof}

Now let us take a look at the special case when the cone $V$ in Proposition~\ref{prop:DS1} splits off $\mathbb{R}^{2n-2}$. Note that when $|\mathrm{Ric}|$ is bounded, we actually have $V = \mathbb{R}^{2n}$ in this case, by Cheeger-Colding-Tian \cite{CCT}. In general
$(V, o)$ is isometric $\mathbb{R}^{2n-2}\times W$, where $(W, o')$ is a two dimensional metric cone. Let us write the metric on $W$ as $dr^2+r^2d\theta^2$, where $0\leq \theta\leq \alpha$ and $\alpha$ is the cone angle of $W$. By \cite{CCT}, the factor $\mathbb{R}^{2n-2}$ has a natural linear complex structure. The conical metric on $W$ also determines a natural complex structure. Thus $(V, o)$ can be identified with $\mathbb{C}^n$. Let the standard holomorphic coordinates be given by $(z_1, .., z_{n-1}, z_n)$, where $z_1, .., z_{n-1}$ are the standard linear coordinates on the first factor $\mathbb{R}^{2n-2}$ and $z_n(r, \theta) = r^{\frac{2\pi}{\alpha}}e^{\frac{2\pi\sqrt{-1}\theta}{\alpha}}$.

Fix $\zeta$ small, and let $s$ be the holomorphic section of $L^m$  constructed in Proposition~\ref{prop:DS1}. For simplicity of notation let us replace $L$ by $L^m$ (i.e. replace $\omega$ by $m\omega$).
Set $h = -\log |s|^2$, so that $\sqrt{-1}\partial\overline\partial h = \omega$. In addition  $h$ is close to the distance squared from $p$.
In particular, if we define $\Omega'$ as the sublevel set $h <1000$, then if $\zeta$ is chosen small, we have $B(p,10) \subset \Omega'\subset\subset B(p, 100)$. Let $\Omega$ be the connected component of $\Omega'$ containing $B(p, 10)$. Then $\Omega$ is a Stein manifold.

Using the same argument as Lemma $4.6$ of \cite{L1} (the proof there only requires the Ricci curvature lower bound), we have the following.
\begin{lemma} We can find $n$ complex harmonic functions $w'_{1}, ..., w'_{n}$ on $B(p, 100)$ so that $w'_{k}$ is $\Psi(\epsilon|n,\nu)$-close to $z_k$ under the Gromov-Hausdorff approximation. Furthermore, $\int_{B(p, 100)}|\overline\partial w'_k|^2\leq \Psi(\epsilon|n, \nu)$.
\end{lemma}

We can solve the $\overline\partial$ problem on $\Omega$ by using the weight $e^{-h}$. By a similar argument to before, we find holomorphic functions $w_{1}, \ldots, w_{n}$ on $B(p, 10)$ which are $\Psi(\epsilon|n,\nu)$-close to the $z_1,\ldots, z_n$. By using the same argument as on page 18 of \cite{L1}, we find that if $\epsilon$ is sufficiently small, $(w_{1}, .., w_{n})$ gives a holomorphic chart on $B(p, 5)$. We have therefore obtained the following result (see also Proposition 12 in \cite{CDS2}).

\begin{prop}\label{prop:codim2cone}
Let $(M, L, \omega)$ be a polarized K\"ahler manifold satisfying $\mathrm{Ric} > -1$ and $\mathrm{vol}(B(p,1)) > \nu$ for all $p\in M$. There exists $\epsilon = \epsilon(n, \nu)$ so that the following holds. Assume that $d_{GH}(B(p, \epsilon^{-1}), B_V(o, \epsilon^{-1}))<\epsilon$, where $(V, o)$ is a metric cone splitting off $\mathbb{R}^{2n-2}$.  Then there exists a holomorphic chart $(w_{1}, ..., w_{n})$ on $B(p, 5)$ such that $(w_{1}, ...., w_{n})$ is $\Psi(\epsilon|n,\nu)$-close to a standard holomorphic coordinate chart on $(V, o)$.
\end{prop}

As a consequence of this we have the following result analogous to Proposition~\ref{prop:scalar}.

\begin{prop}\label{prop:cone2}
  Suppose that $(M_i^n, L_i, \omega_i)$ is a sequence of polarized K\"ahler manifolds with $\mathrm{Ric} > -1$, $\mathrm{vol}(B(q_i,1)) > \nu > 0$ for all $q_i\in M_i$. For $p_i\in M_i$, assume that $(M_i, p_i)$ converges to a metric cone $(V,o)$ in the pointed Gromov-Hausdorff sense, where $V=\mathbb{R}^{2n-2}\times W$ with $(W,o')$ a two-dimensional cone. Then
  \[ \lim_{i\to\infty} \int_{B(p_i,1)} S_i \omega_i^n = \omega_{2n-2} (2\pi -\alpha), \]
where $\omega_{2n-2}$ is the area of the unit ball in $\mathbb{R}^{2n-2}$ and $\alpha\in(0,2\pi)$ is the cone angle of $W$. Note that the distributional scalar curvature of $W$ is $(2\pi-\alpha)\delta_{o'}$.
\end{prop}
\begin{proof}
  The proof of Proposition~\ref{prop:scalar} can be used essentially verbatim, using that under our assumptions Proposition~\ref{prop:codim2cone} gives suitable holomorphic charts on $B(p_i,10)$. The main difference is that now on the limit space the function $|s|=|dz_1\wedge\ldots\wedge dz_n|$ vanishes along the singular set, and so $\log |s|$ is unbounded. Instead of the statement of Lemma~\ref{lm1}, we have that for any neighborhood $U$ of the singular set (identified with a subset of $B(p_i, 4)$ using the chart),
  \[ \lim\limits_{i\to\infty}\int_{B(p_i, 4) \backslash U}\left|\log |s_i|^2-\log |s|^2\right|\omega_i^n = 0.\]
  Then just as in \eqref{-13} we will have
  \begin{equation}\label{eq:c1}\lim_{i\to\infty} \int_{M_i \setminus U} \log |s_i|^2 \Delta_{\omega_i} h = \int_{V\setminus U} \log |s|^2\Delta_\omega h.
\end{equation}
Notice that on sufficiently small neighborhoods $U$ of the singular set, the integrals of both $\log |s|$ and $\log |s_i|$ can be made arbitrarily small. The former by direct calculation, and the latter by the estimate \eqref{eq:b1}. It then follows from \eqref{eq:c1} that
\[ \lim_{i\to\infty} \int_{M_i} \log |s_i|^2 \Delta_{\omega_i} h = \int_{V} \log |s|^2\Delta_\omega h, \]
which implies the required result.
\end{proof}

Using this result, together with Cheeger-Jiang-Naber's~\cite{CJN} bounds we now prove Proposition~\ref{prop:Sintbdd}, which we state again for the reader's convenience.
\begin{prop}
Let $B(p,1)$ be a unit ball in a polarized K\"ahler manifold $(M^n,L,\omega)$ satisfying $\mathrm{Ric} > -1$, such that $\mathrm{vol}(B(p,1)) > v > 0$. Then $\int_{B(p,1)} S < C(n,v)$.
\end{prop}
\begin{proof}
  For any $l >0$, let $A_l$ denote the supremum of $\int_{B(p,1)} |S|$ over all unit balls as in the statement, with the additional condition that $\mathrm{Ric} < l$. Our goal is to show that $A_l$ is bounded independently of $l$. Note that any $A_l$ is finite by volume comparison. Also it is convenient to replace the condition $\mathrm{vol}(B(p,1)) > v > 0$ by
  \[ r^{-2n}\mathrm{vol}(B(q,r)) > v' > 0, \text{ for all $q\in B(p,1)$ and $r < 1$,} \]
  since this condition is preserved when passing to smaller balls. The two conditions imply each other for suitable $v,v'$ by volume comparison.

Let us recall the following notion from Cheeger-Jiang-Naber~\cite[Definition 1.3]{CJN}. A ball $B(x,r)$ in a metric space is $(k,\epsilon)$-symmetric if there is a metric cone $X' = \mathbf{R}^k\times C(Z)$ with vertex $x'$ splitting an isometric factor of $\mathbf{R}^k$, such that $d_{GH}(B(x,r), B(x',r)) < \epsilon r$.
From Corollary~\ref{cor:Sintsmall} and Proposition~\ref{prop:cone2} we find that there are constants $\epsilon, C_1 > 0$ with the following property. If $B(q,\epsilon^{-1})$ is a ball in a polarized K\"ahler manifold with $\mathrm{Ric} > -1, \mathrm{vol}(B(q,1)) > v$, and $B(q,\epsilon^{-1})$ is $(2n-2, \epsilon^2)$-symmetric, then $\int_{B(q,1)} |S|< C_1$.

Let $r > 0$ be small, to be chosen later, and set $k=2n-3$. As in \cite{CJN}, let $S^{2n-3}_{\epsilon^2, r}$ denote the points $x\in B(p,1)$ such that $B(x,s)$ is not $(2n-2,\epsilon^2)$-symmetric for any $s\in [r,1)$. Let us choose $x_1,\ldots, x_{N_r} \in S^{2n-3}_{\epsilon^2, r}$ such that $B(x_i, r)$ cover $S^{2n-3}_{\epsilon^2, r}$, while $B(x_i, r/3)$ are disjoint.
By \cite[Remark 1.10]{CJN} we have that $N_rr^{2n-3} \leq C_\epsilon$. In addition, if $\mathrm{Ric} < l$, then after scaling the ball $B(x_i,r)$ to unit size, it will still have Ricci curvature bounded by $l$. We can assume that $r^{-2}$ is an integer so the scaled up manifold is still polarized.  It follows by scaling that
\[ \int_{B(x_i, r)} |S| < r^{2n-2}A_{l}. \]

If $y\not\in \bigcup B(x_i,r)$, then by definition there is an $s\in [r, 1)$ such that $B(y, s)$ is $(2n-2, \epsilon^2)$-symmetric. It follows, after rescaling the result above that
\[ \int_{B(y, \epsilon s)} |S| < (\epsilon s)^{2n-2} C_1. \]
We can now cover $B(p,1) \setminus \bigcup B(x_i,r)$ by such balls $B(y_j, \epsilon s_j)$, such that the $B(y_j, \epsilon s_j/5)$ are disjoint. We then have
\[ \sum_j (\epsilon s_j)^{2n} < C_2, \]
and
\[ \int_{B(p,1)\setminus \bigcup_i B(x_i, r)} |S| \leq \sum_j \int_{B(y_j, \epsilon s_j)} |S| \leq C_1 \sum_j (\epsilon s_j)^{2n-2} < C_1C_2 (\epsilon r)^{-2}, \]
using $s_j \geq r$. In sum we get
\[ \int_{B(p,1)}  |S| < N_r r^{2n-2}A_{l} + C_1C_2 (\epsilon r)^{-2} \leq r C_\epsilon A_{l} + C_1C_2 (\epsilon r)^{-2}. \]
We now choose $r$ so that $r C_\epsilon < 1/2$. It follows that
\[ A_l \leq \frac{1}{2} A_l + C',\]
where $C'$ is independent of $l$. This implies our result.
\end{proof}

Let us now return to the setting of the beginning of the section, i.e. $(M_i^n, L_i, \omega_i)$ are polarized K\"ahler manifolds with $\mathrm{Ric}(\omega_i) > -\omega_i, \mathrm{diam}(M_i) < d, \mathrm{vol}(M_i) < v$, converging to $X$ in the Gromov-Hausdorff sense.
Given Proposition~\ref{prop:DS1}, an argument by contradiction implies the partial $C^0$-estimate and separation of points:
\begin{prop}
Given any point $p\in X$, take a sequence $M_i\ni p_i\to p$. There exist $\delta = \delta(n, v, d)>0$,  $K(n, v, d)\in\mathbb{N}$ and holomorphic sections $s_i$ over $L_i^m (m<K(n, v, d))$ so that
$\int |s_i|^2 = 1$, $|s_i(p_i)| \geq \delta$.

Furthermore, given any two points $p, q\in X$ with $d(p, q)>a>0$ and sequences $M_i\ni p_i\to p, M_i\ni q_i\to q$, we can find holomorphic sections $s_i^1, s_i^2$ of $L^m (m<K(n, v, a, d))$ and $\delta=\delta(n, v, a, d)$ so that
\begin{itemize}
\item $\int |s_i^1|^2 + |s_i^2|^2<1$;
\item $|s_i^1(p_i)| = \delta, s_i^1(q_i) = 0$;
\item $s_i^2(p_i) = 0, |s_i^2(q_i)| = \delta$.
\end{itemize}
\end{prop}

By following the same arguments as in \cite[Section 4.3.1]{DS1} we can prove that $X$ is homeomorphic to a projective variety and after suitable projective embeddings a subsequence of the $M_i$ converge to $X$ as algebraic varieties. In addition, Proposition~\ref{prop:codim2cone} implies that $X$ is complex analytically regular near the points $p\in X$ which admit tangent cones splitting off $\mathbb{R}^{2n-2}$. The remainder of $X$ has Hausdorff dimension at most $2n-4$ by \cite{[CC2]}, which implies as in \cite{DS1} that $X$ is normal. This completes the proof of Theorem~\ref{thm:main}.

\section{Complex analytic and metric singularities}\label{sec:singularities}
As in the previous section, given $n, d, v$, let $(M_i^n, \omega_i, L_i)$ be polarized K\"ahler manifolds with $\mathrm{Ric}(M_i) > -1, \mathrm{diam}(M_i) < d, \mathrm{vol}(M_i) > v$, such that $M_i$ converge in the Gromov-Hausdorff sense to $X$. Then $X$ has the structure of a normal projective variety, and it is also a metric space. In this section we will study the relation between the singular sets of $X$ in the metric sense and in the complex analytic sense.

Pick a point $p\in X$, and take a sequence $M_i\ni p_i\to p$. As a consequence of Proposition~\ref{prop:DS1}, for sufficiently large $i$, there exists $r_0>0$ independent of $i$ and a smooth $u_i$ on $B(p_i, r_0)$ so that $\sqrt{-1}\partial\overline\partial u_i = \omega_i$ and $u_i(x)$ is close to $d^2(x, p_i)$. Since $\mathrm{Ric}(\omega_i) > -\omega_i$, $\Theta_i = \mathrm{Ric}(\omega_i)+\sqrt{-1}\partial\overline\partial u_i$ is a closed positive current on $B(p_i, r_0)$. Now assume that $p$ is a complex analytically regular point on $X$(i.e., a non-singular point on the variety $X$). By shrinking the value of $r_0$ if necessary, by solving the $\bar\partial$-problem with weights $e^{-u_i}$, we can find uniform holomorphic charts $(z^i_{1}, ..., z^i_{n})$ on $B(p_i, r_0)$ so that the charts converge to a holomorphic chart $(U, (z_1, .., z_n))$ near $p$. As before, we can use these charts to identify the balls $B(p_i, r_0/2)$ with corresponding subsets of $U$. Let us say $z_j(p) = 0$ for all $j = 1, ..., n$, and define
\[ v_i = \frac{1}{2\pi}\log |dz^i_{1}\wedge dz^i_{2}\wedge\ldots\wedge dz^i_{n}|^2 + u_i.\]
Note that $\Theta_i = \sqrt{-1}\partial\overline\partial v_i\geq 0$. Since $U$ is an open set, $|dz^i_{1}\wedge dz^i_{2}\wedge\ldots\wedge dz^i_{n}|$ cannot go to zero uniformly, as $i\to\infty$. Thus $v_i$ cannot go uniformly to $-\infty$. By taking a subsequence, we may assume that $v_i$ converges, in $L^1_{loc}$ sense (with respect to the Lebesgue measure given by the charts), to a plurisubharmonic function $v$ on $U$. As $u_i$ has uniformly bounded gradient (note that $\Delta u_i = 2n$), we can also assume that $u_i\to u$ uniformly. Define
\begin{equation}\label{-15}\mathrm{Ric} =  \sqrt{-1}\partial\overline\partial (v-u).\end{equation}
Then $\mathrm{Ric}$, as a closed $(1, 1)$ current, is well defined on the complex analytically regular part of $X$. Note that $\mathrm{Ric}$ is a closed positive current, up to $\sqrt{-1}\partial\overline\partial u$ for a bounded function $u$. Thus the Lelong number for $\mathrm{Ric}$, given by
\begin{equation}\label{-16}\frac{1}{2\pi}\liminf_{x\to p}\frac{\log |dz_{1}\wedge dz_{2}\wedge\ldots\wedge dz_{n}|^2}{\log |z(x)|},\end{equation}
is well defined. Here the numerator can be defined as the limit of $\log |dz^i_{1}\wedge dz^i_{2}\wedge\ldots\wedge dz^i_{n}|^2$.

The main result in this section is the following
\begin{prop}\label{prop-10}
A point $p\in X$ is regular in the metric sense if and only if it is complex analytically regular and the Lelong number for $\mathrm{Ric}$ vanishes at $p$.
\end{prop}
\begin{proof}
  It is clear from Theorem~\ref{thm:chart} that if $p$ is regular in the metric sense, then $p$ must be complex analytically regular. Now let us prove that if $X$ is complex analytically regular at $p$ and the Lelong number for $\mathrm{Ric}$ is zero at $p$, then $p$ is a regular point in $X$ in the metric sense. We first need some preliminary results.

\begin{claim}\label{cl-10}
There exists $a>0$, $b=b(n, v, d)>0$ so that for all $r<a$, and any point $q\in \partial B(p, r)$, there exists a holomorphic function $f$ on $B(p, 4r)$ so that $f(p) = 0$, $\sup\limits_{B(p, 2r)} |f| = 1$ and $|f(q)|>b$.
\end{claim}
\begin{proof}
Assume that the claim is false. Then there exist sequences $r_i, b_i\to 0$ and $q_i\in \partial B(p, r_i)$ so that for all holomorphic functions $f_i$ on $B(p, 4r_i)$ with $f_i(p) = 0$ and $\sup\limits_{B(p, 2r_i)}|f_i| =1$, $|f_i(q_i)|<b_i$. By passing to a subsequence, we may assume that $(X_i, p_i, d_i)= (X, p, \frac{d}{r_i})$ converges to a metric cone $(V, o)$ in the pointed Gromov-Hausdorff sense. Assume $q_i\to q\in \partial B(o, 1)$.

To get a contradiction, we can prove results similar to Theorem $1.4$ and Proposition 2.9 of \cite{DS2} (alternatively, Proposition $6.1$ in \cite{L2}). The proof follows by a very minor modification, so we skip the details. Then on $(V,o)$ we can find a holomorphic function vanishing at $o$ but nonzero at $q$, which we can lift to $(X_i, p_i)$ for sufficiently large $i$. The lifted holomorphic functions will have a uniform lower bound at $q_i$, giving a contradiction. It is clear from the argument that $b$ depends only on $n, v, d$.
\end{proof}

\begin{claim}\label{cl-11}
Let $p$ be a complex analytically regular point on $X$. Let $(z_1, ..., z_n)$ be a holomorphic chart near $p$, such that $z_j(p) = 0$ for all $j$. Then there exists $\alpha = \alpha(n, v, d)>0$, $C>0, c>0$ so that $cr(q)^\alpha\leq |z(q)|\leq Cr(q)$ for all $q$ sufficiently close to $p$. Here $r$ is the distance function to $p$.
\end{claim}
\begin{proof}
The inequality $|z(q)|\leq C r(q)$ follows directly from the gradient estimate. Now we prove the first inequality. Let $a>0, b = b(n, v, d)>0$ be the constants in Claim \ref{cl-10}. Let us fix a small $r_0<a$. We may assume that $B(p, 2r_0)$ is contained in the holomorphic chart $(z_1, ..., z_n)$. For any $\rho>0$, let $U_\rho$ be the open set so that $|z|<\rho$. Since $(z_1, .., z_n)$ is a holomorphic chart, for $\delta$ sufficiently small, we may assume that $U_\delta\subset B(p, 2r_0)$. Pick an arbitrary $q\in \partial B(p, r_0)$. According to Claim \ref{cl-10}, there exists a holomorphic function $f$ on $B(p, 4r_0)$ so that $f(p) = 0$, $\sup\limits_{B(p, 2r_0)}|f| = 1$, $|f(q)|>b$.
If $q \in U_\delta$, we restrict $f$ to $U_\delta$. As $f(p) = 0$, by the standard Hadamard three circle theorem on $B_{\mathbb{C}^n}(0, \delta)$ we find $\frac{|f(q)|}{|z(q)|}\leq \frac{\sup_{U_\delta}|f|}{\delta}$, thus $|z(q)|\geq b\delta$. If $q\notin U_\delta$, then by definition $|z(q)|\geq b\delta$. Since $q$ is arbitrary on $\partial B(p, r_0)$, we find that $U_{b\delta}\subset B(p, r_0)$. Iterating this result, we obtain that $U_{b^k\delta} \subset B(p, 2^{1-k}r_0)$, which implies the first inequality $c r(q)^\alpha \leq |z(q)|$.
\end{proof}

\begin{claim}\label{cl-12}Assume that $p$ is not a regular point in the metric sense. Then
  there exist $\epsilon>0$ and $r_0>0$ satisfying the following. For all $r<r_0$, if nonzero holomorphic functions $f_1, .., f_n$ on $B(p, 4r)$ satisfy $f_j(p) = 0$ and $\int_{B(p, r)}f_j\overline{f}_k = 0$ for $j\neq k$,  then there exists $1\leq l\leq n$ so that
  \[ \frac{\dashint_{B(p, 2r)}|f_l|^2}{\dashint_{B(p, r)}|f_l|^2}\geq 2^{2+10n\epsilon}. \]
\end{claim}
\begin{remark}
From the proof it follows that $\epsilon$ depends only on $\omega_{2n}- \lim\limits_{r\to 0}\frac{\mathrm{vol}(B(p, r))}{r^{2n}}$, where $\omega_{2n}$ is the volume of the unit ball of $\mathbb{C}^n$.
\end{remark}
\begin{proof}
If this is not the case, then we can find  sequences $\epsilon_i\to 0$, $r_i\to 0$, and nonzero holomorphic functions $f_{i1}, ..., f_{in}$ on $B(p, 4r_i)$ so that $f_{ij}(p) = 0$ and $\int_{B(p, r_i)}f_{ij}\overline{f}_{ik} = 0$ for $j\neq k$. Also for all $j$, $$\frac{\dashint_{B(p, 2r_i)}|f_{ij}|^2}{\dashint_{B(p, r_i)}|f_{ij}|^2}< 2^{2+10n\epsilon_i}.$$
Define $(X_i, p_i, d_i) = (X, p, \frac{d}{r_i})$. By passing to a subsequence, we may assume that $(X_i, p_i)$ converges in the pointed Gromov-Hausdorff sense to a tangent cone $(V, o)$ at $p$.
We trivially lift $f_{ij}$ to $B(p_i, 4)$ on $X_i$. By normalization, we may assume that  $\dashint_{B(p_i, 1)}|f_{ij}|^2 = 1$ for all $j$.  Then after taking a further subsequence, $f_{ij}$ converges uniformly on each compact set of $B(o, 2)$ to linearly independent complex harmonic functions $h_j$. These satisfy $\dashint_{B(o, 1)}|h_j|^2 = 1$, $h_j(o) = 0$ for all $j$, and $\dashint_{B(o, 2)}|h_j|^2\leq 4$. By the spectral decomposition for the Laplacian on the cross section, the $h_j$ can be extended as degree one homogeneous complex harmonic functions on $V$. Therefore we have $2n$ linearly independent real harmonic functions of linear growth which all vanish at $o$. Then it is well known that $(V, o)$ is isometric to $\mathbb{R}^{2n}$ (see Proposition~\ref{prop:splitting} in the Appendix). This contradicts the assumption that $p$ is not a regular point.
\end{proof}

Now let $p\in X$ be a regular point in the complex analytic sense, but not in the metric sense. We claim that the Lelong number of $\mathrm{Ric}$ at $p$ is positive.
Let $(z_1, .., z_n)$ be a holomorphic chart around $p$. For $r_0$ small, we may assume $B(p, 2r_0)$ is contained in the chart. By suitable scaling and orthogonalization of $z_j$, we may assume \begin{equation}\label{-17} z_j(p)=0, \quad \dashint_{B(p, 2r_0)} |z_j|^2= 1, \quad (j=1, ..., n),
\end{equation}
\begin{equation}\label{-18}\int_{B(p, 2r_0)}z_j\overline{z}_k = 0, \text{ for }j\neq k.\end{equation}
By scaling  the metric, let us assume without loss of generality that $r_0 = 1$. Then by the gradient estimate
\begin{equation}\label{-19}|dz_1\wedge\ldots\wedge dz_n|\leq C=C(n, v, d) \text{ on } B(p,1). \end{equation}
Let $E$ be the linear space spanned by $z_1, ..., z_n$. On $E$, we have two norms, given by $L^2$ integration over $B(p, 2)$ and $B(p, 1)$. After a simultaneous diagonalization with respect to these two norms, we may assume that $z_j$ are also $L^2$ orthogonal on $B(p, 1)$. Let $\epsilon$ be the positive number appearing in Claim \ref{cl-12}. We may assume that $B(p, 10)$ is sufficiently close to a metric  cone, such that \begin{equation} \label{eq:22}
  \frac{\dashint_{B(p, 2)}|z_j|^2}{\dashint_{B(p, 1)}|z_j|^2}\geq 2^{2-\epsilon}
\end{equation}
for all $j$, since on the cone there are no non-constant sublinear harmonic functions. According to Claim \ref{cl-12}, we can find $l$ so that
$$\frac{\dashint_{B(p, 2)}|z_l|^2}{\dashint_{B(p, 1)}|z_l|^2}\geq 2^{2+10n\epsilon}.$$
Define $z'_j = z_j2^{-\epsilon}$ for $j\neq l$; $z'_l = z_l 2^{(n-1)\epsilon}$, so that \[dz'_1\wedge dz'_2\wedge\ldots\wedge dz'_n = dz_1\wedge dz_2\wedge\ldots\wedge dz_n\]
at any point. Moreover, from \eqref{eq:22} we have
\[ \dashint_{B(p, 1)} |2z'_k|^2\leq 2^{-\epsilon}\]
for all $k$. Using the gradient estimate we obtain that on $B(p, \frac{1}{2})$, $$|dz_1\wedge dz_2\wedge \ldots\wedge dz_n| = |dz'_1\wedge dz'_2\wedge \ldots\wedge dz'_n| \leq C2^{-0.5n\epsilon}\leq C2^{-\epsilon}.$$  Here $C$ is the same constant as in (\ref{-19}). By iteration, we obtain that for all $0<r<1$, on $B(p, r)$, $$|dz_1\wedge dz_2\wedge \ldots\wedge dz_n|\leq 2Cr^{\epsilon}.$$

   According to Claim \ref{cl-11}, Claim \ref{cl-12}, the Lelong number for $2\pi\mathrm{Ric}$ at $p$ is given by $$\lim\inf_{x\to p}\frac{\log |dz_1\wedge dz_2\wedge \ldots\wedge dz_n|^2(x)}{\log |z(x)|}\geq \lim\inf_{x\to p}\frac{2\log (2Cr(x)^\epsilon)}{\log(cr(x)^\alpha)} = \frac{2\epsilon}{\alpha}>0.$$
 This means that if $p$ is complex analytically regular and the Lelong number for $\mathrm{Ric}$ vanishes at $p$, then $p$ is regular in the metric sense.

 \bigskip

We are left to prove that if $p$ is regular in the metric sense, then the Lelong number for $\mathrm{Ric}$ vanishes at $p$.
Since $p$ is metric regular, for any fixed small $\epsilon>0$, by scaling, we may assume that $B(p, \frac{1}{\epsilon})$ is $\epsilon$-Gromov Hausdorff close to a ball in $\mathbb{C}^n$. Then we can find a holomorphic map $F=(f_1, ..., f_n)$ on $B(p, 100)$ which gives a $\Psi(\epsilon|n)$-Gromov Hausdorff approximation to its image in $\mathbb{C}^n$. As in the proof of Theorem~\ref{thm:chart}, $F$ is a holomorphic chart on $B(p, 1)$. Without loss of generality, we may assume $f_j(p) = 0, \int_{B(p, 1)}f_j\overline{f_k} = 0$ for $j\neq k$.
We have
\begin{equation}\label{-100}\frac{\dashint_{B(p, 2)}|f_j|^2}{\dashint_{B(p, 1)}|f_j|^2}\leq 4+\Psi(\epsilon|n),\end{equation}
and since $B(p, r)$ is $\Phi(\epsilon|n)r$-Gromov-Hausdorff close to a Euclidean ball for all $r>0$, just as in Lemma~\ref{lem:3annulus} we conclude that \begin{equation}\label{eq:101} 4-\Psi(\epsilon|n)\leq\frac{\dashint_{B(p, 2r)}|f_j|^2}{\dashint_{B(p, r)}|f_j|^2}\leq 4+\Psi(\epsilon|n)
\end{equation}
for all $r<1$. For any $r>0$, we may assume that the $f_1, ..., f_n$ are orthogonal simultaneously with respect to the $L^2$ inner products on $B(p, 1)$ and $B(p, r)$.
Now let $c_j$ be constants (depending on $r$) so that $\sup\limits_{B(p, r)}|c_jf_j| = r$. Define $f'_j = c_jf_j$.
As in the proof of Proposition~\ref{prop:gapthm}, arguing as in Claim $4.2$ of \cite{L2}, we see that $F' = (f'_1, .., f'_n)$ is a $\Psi(\epsilon|n)r$-Gromov-Hausdorff approximation to a ball in $\mathbb{C}^n$. The Cheeger-Colding~\cite{CC} estimate (see Equation~\eqref{eq:CCestimate}) implies that
$$\sup\limits_{B(p, r)}|df'_1\wedge df'_2\wedge\ldots\wedge df'_n|\geq c(n)>0,$$
and note that by (\ref{eq:101}),
$$\dashint_{B(p, 2r)}|f_j|^2\geq c(n)r^{2+\Psi(\epsilon|n)}.$$
Thus
$$|c_j|\leq C(n)r^{-\Psi(\epsilon|n)},$$
and so
$$\sup\limits_{B(p, r)}|df_1\wedge df_2\wedge\ldots\wedge df_n|\geq c(n)r^{\Psi(\epsilon|n)}.$$
It follows that the Lelong number at $p$ satisfies
$$\lim\inf_{x\to p}\frac{\log |df_1\wedge df_2\wedge \ldots\wedge df_n|^2(x)}{2\pi\log |F(x)|}\leq \lim\inf_{x\to p}\frac{2\log(c(n)r(x)^{\Psi(\epsilon|n)})}{2\pi\log (Cr(x))}=\Psi(\epsilon|n).$$
As $\epsilon$ is arbitrary, we find that the Lelong number at $p$ is zero. The proof of Proposition \ref{prop-10} is complete.
\end{proof}

Let $A$ be the complex analytically singular set of $X$. Let $c$ be a positive constant, and let $H_c$ be set of points whose Lelong number for $\mathrm{Ric}$ is at least $c$ on $X\backslash A$. By Siu's theorem \cite{S2}, $H_c$ is a complex analytic set of $X\backslash A$. We thank Professor Siu for providing the proof of the following.
\begin{lemma}\label{SiuLemma}
The topological closure of $H_c$ in $X$ is a complex analytic set.
\end{lemma}
\begin{proof}The problem is local on $X$.
For any $p\in X$, take a sequence $M_i\ni p_i\to p$. For sufficiently large $i$, there exists $r_0>0$ independent of $i$ and smooth functions $u_i$ on $B(p_i, 2r_0)$ so that $\sqrt{-1}\partial\overline\partial u_i = \omega_i$ and $u_i(x)$ is close to $d^2(x, p_i)$. Write $U = B(p, r_0)$ and assume $u_i\to u$ uniformly on $U$. Now assume that $p$ is a complex analytically singular point on $X$.  Then $\Theta = \mathrm{Ric} + \sqrt{-1}\partial\overline\partial u$ is a closed positive $(1, 1)$ current on $U\backslash A$. It is clear that the Lelong number for $\Theta$ is the same as Lelong number for $\mathrm{Ric}$. By shrinking $U$ if necessary, we may assume that $(U, p)$ is a normal subvariety of $(\Omega, 0)\subset(\mathbb{C}^N, 0)$. We can trivially extend $\Theta$ as a positive closed $(N-n+1, N-n+1)$-current $\hat{\Theta}$ on $\Omega\backslash A$. Since $U$ is a normal variety, $A$ has complex dimension at most $n-2$. Thus the codimension of $A$ in $\Omega$ is at least $N-n+2$. According to \cite{S1}, $\hat{\Theta}$ extends to a closed positive current on $\Omega$. By applying Siu's theorem again, we proved the lemma.
\end{proof}

The following is a generalization of Donaldson-Sun's Proposition 4.14 in \cite{DS1}, where the Einstein case was treated (although their proof applies in the case of bounded Ricci curvature too).
\begin{cor}\label{cor:DSsingular}
Suppose that the $M_i$ above have uniformly bounded Ricci curvature. Then the metric singular set coincides with the complex analytic singular set on $X$.
\end{cor}
 \begin{proof}
Let $p$ be a complex analytically regular point. It suffices to prove that $p$ is metric regular.
According to Proposition~\ref{prop-10}, it suffices to show that the Lelong number for $\mathrm{Ric}$ vanishes at $p$. Recall the bounded function $u$ in the last lemma.
Let us assume $|\mathrm{Ric}(M_i)|\leq C$. Then as a closed positive $(1, 1)$ current, $\Theta = \mathrm{Ric}+ C\sqrt{-1}\partial\overline\partial u \leq 2C\sqrt{-1}\partial\overline\partial u$.
Note that the Lelong number for $\Theta$ is the same as for $\mathrm{Ric}$ at $p$. By monotonicity, the Lelong number of $\Theta$ is no greater than the Lelong number of the positive $(1, 1)$-current $2C\sqrt{-1}\partial\overline\partial u$. But $u$ is bounded, therefore the Lelong number for $\mathrm{Ric}$ vanishes at $p$.

 \end{proof}

 Theorem \ref{thm:singularset} is a direct consequence of Proposition \ref{prop-10}. For the reader's convenience, we rewrite it here.
\begin{theorem}
 Let $(X,d)$ be a Gromov-Hausdorff limit as in Theorem~\ref{thm:main}.
 Then for any $\epsilon>0$, $X\setminus \mathcal{R}_\epsilon$ is contained in a finite union of analytic subvarieties of $X$. Furthermore, $X\setminus \mathcal{R}$ is equal to a countable union of subvarieties.
\end{theorem}
Here $\mathcal{R}$ consists of points $x\in X$ with tangent cone $\mathbb{C}^n$, while $\mathcal{R}_\epsilon$ is the set of points $p$ so that $\omega_{2n}-\lim\limits_{r\to 0}\frac{\mathrm{Vol}(B(p, r))}{r^{2n}}  < \epsilon$..
In view of the main result of \cite{L1}, we can use the same argument to obtain the following.
\begin{theorem}\label{thm:bisect}
  Let $(X, p)$ be the pointed Gromov-Hausdorff limit of complete K\"ahler manifolds $(M^n_i, p_i)$ with bisectional curvature lower bound $-1$ and $\mathrm{vol}(B(p_i, 1))\geq v>0$. Then $X$ is homeomorphic to a normal complex analytic space. The metric singular set $X\setminus\mathcal{R}$ is exactly given by a countable union of complex analytic sets, and for any $\epsilon>0$, each compact subset of $X\setminus\mathcal{R}_\epsilon$ is contained in a finite union of subvarieties.\end{theorem}

\section{Appendix}
In the proof of Proposition~\ref{prop:DS1} we used the estimate of Cheeger-Jiang-Naber~\cite{CJN} for the volumes of tubular neighborhoods of the singular set, in order to control the cutoff functions $\psi^3_i$. In this appendix we first give an alternative argument, following the approach of Chen-Donaldson-Sun~\cite{CDS2}, which is independent of the results in \cite{CJN}.
First note that the cone $V = \mathbb{R}^{2n-2} \times W$ for a two-dimensional cone $(W,o')$ has singular set $\mathbb{R}^{2n-2} \times \{o'\}$, and so we can directly see the required estimate for the volumes of its tubular neighborhoods. Therefore Propositions~\ref{prop:codim2cone} and \ref{prop:cone2} hold without appealing to \cite{CJN}.

We can now argue similarly to \cite{CDS2} to show that the singular sets in any cone $(V,o)$ arising in Proposition~\ref{prop:DS1} have good cutoff functions  (even closer to what we do are Propositions 12, 13 and 14 of the arXiv version of \cite{Sz}). More precisely, suppose that $(M_i, L_i, \omega_i)$ is a sequence as in Proposition~\ref{prop:DS1} such that $(M_i, p_i)$ converges to a cone $(V,o)$ for some $p_i\in M_i$. Recall that the singular set is $\Sigma\subset V$, consisting of points $q\in V$ such that $\lim_{r\to 0} r^{-2n}\mathrm{vol}(B(q,r)) \leq \omega_{2n}-\epsilon$, where $\epsilon$ is obtained from Theorem~\ref{thm:chart}. We then have the following.
\begin{prop} For any compact set $K\subset V$ and $\kappa > 0$ there is a function $\chi$ on $V$, equal to 1 on a neighborhood of $K\cap \Sigma$, supported in the $\kappa$-neighborhood of $K\cap \Sigma$, and such that $\int_K |\nabla\chi|^2  < \kappa$.
\end{prop}

\begin{proof}
  Let us fix $K, \kappa$.
Suppose that we have distance functions $d_i$ on $B(p_i,2R) \sqcup B(o,2R)$, realizing the Gromov-Hausdorff convergence, where $R$ is large so that $K\subset B(o,R)$. For $q\in B(o,R)$, and $\rho\in (0,1)$, define
\[ V(i, q, \rho) = \rho^{2-2n} \int_{U_i(q, \rho)} (S_i + 2n) \omega_i^n, \]
where $U_i(q,\rho) = \{x \in B(p_i,2)\,:\, d_i(x, q) < \rho\}$. Note that $S_i + 2n \geq 0$.

Let us denote by $\mathcal{D}\subset K\cap \Sigma$ the set of points which have a tangent cone splitting off $\mathbb{R}^{2n-2}$, and let $S_2 = \Sigma \setminus \mathcal{D}$. Proposition \ref{prop:cone2} implies that for any $x\in \mathcal{D}$ there exists a $\rho_x > 0$ such that $V(i,x,\rho_x) < A$ for a fixed constant $A$, for sufficiently large $i$. At the same time by Proposition~\ref{prop:cone2} we also have a constant $c_0 > 0$ (depending on $n, v, d, \epsilon$) such that for any $x\in \mathcal{D}$ and $\delta > 0$ there is an $r_x < \delta$ such that $V(i,x,r_x) > c_0$ for sufficiently large $i$ (here note that the two dimensional cones appearing in tangent cones of points in $\mathcal{D}$ have cone angles bounded strictly away from $2\pi$).

By Cheeger-Colding's \cite{[CC2]} the Hausdorff dimension of
$S_2$ is at most $2n-4$, so for any small $\epsilon > 0$
 we can cover $S_2\cap K$ with balls $B_\mu$ such that
\[ \sum_\mu r_\mu^{2n-3} < \epsilon. \]
The set
\[ J = (K\cap \Sigma) \setminus \cup_\mu B_\mu  \]
is compact, $J\subset \mathcal{D}$, and so it is
covered by the balls $B(x,\rho_x)$ with $x\in \mathcal{D}$. We choose a
finite subcover corresponding to $x_1,\ldots, x_N$, and set
\[ \begin{aligned}
    W &= \bigcup_{j=1}^N B(x_j, \rho_{x_j}) \subset V, \\
    W_i &= \bigcup_{j=1}^N U_i(x_j, \rho_{x_j}) \subset M_i.
\end{aligned} \]
For sufficiently large $i$ we then get an estimate
\begin{equation}\label{eq:1}
\int_{W_i} (S_i + 2n) \omega_i^n < C,
\end{equation}
where $C$ depends on $\epsilon, N$, but not on $i$.

We claim that the compact set $ J \subset\mathcal{D}$ has finite $(2n-2)$-dimensional
Hausdorff measure. To prove this, recall that for any small $\delta >
0$ and all $x\in
\mathcal{D}\cap J$ we have  $r_x < \delta$ such that $V(i,x,r_x) >
c_0$ for large $i$. By a Vitali type covering argument we can find a
disjoint, finite sequence of balls $B(y_k, r_{y_k})$ in $W$, for $k=1,\ldots, N'$
such that $B(y_k, 5r_{y_k})$ cover all of $J$. It follows that
\[ \mathcal{H}^{2n-2}_\delta (J) \leq \sum_{k=1}^{N'} 5^{2n-2}
r_{y_k}^{2n-2}. \]
At the same time for each $y_k$, we have the estimate
\[ r_{y_k}^{2-2n} \int_{U_i(y_k,r_{y_k})} (S_i + 2n)
\omega_i^n > c_0, \]
for sufficiently large $i$. Taking $i$ even larger we can assume that the $U_i(y_k, r_{y_k})$ are disjoint, since they converge in the Gromov-Hausdorff sense to the disjoint balls $B(y_k, r_{y_k})$. Using \eqref{eq:1} we therefore have
\[ \sum_{k=1}^{N'} c_0 r_{y_k}^{2n-2} < C. \]
Since $\delta$ was arbitrary (and $C$ is independent of $\delta$),
this implies that  $\mathcal{H}^{2n-2}(J) \leq C'$.

It follows that $ J$ has capacity zero, in the
sense that for any $\kappa > 0$ we can find a cutoff function
$\eta_1$ supported in the $\kappa$-neighborhood of $J$, such that $\Vert \nabla \eta_1\Vert_{L^2} \leq \kappa /2$, and
$\eta_1 = 1$ on a neighborhood $U$ of $J$ (see for
instance \cite[Lemma 2.2]{Bou} or \cite[Theorem 3, p. 154]{EG92}). The set
$(K\cap \Sigma) \setminus U$ is compact, and so it is covered by
finitely many of our balls $B_\mu$ from before. Because of this, as in
\cite{DS1},
we can find a good cutoff function
$\eta_2$, with $\Vert \nabla \eta_2\Vert_{L^2} \leq \kappa/2$ (if
$\epsilon$ at the beginning was sufficiently small), such that $\eta_2$ is supported in the $\kappa$-neighborhood of $(K\cap\Sigma)\setminus U$ and with
$\eta_2 = 1$ on a neighborhood of $(K\cap \Sigma)\setminus U$. Then
$\eta = 1 - (1-\eta_1)(1-\eta_2)$ gives the required cutoff function.
\end{proof}

We next prove a result essentially contained in Cheeger-Colding-Minicozzi~\cite{CCM}, that we used in the proof of Proposition~\ref{prop-10}.
\begin{prop}\label{prop:splitting}
  Let $(V,o)$ denote a tangent cone of a non-collapsed limit space of manifolds with Ricci curvature bounded from below.
  Suppose that there are $k$ linearly independent harmonic functions
  $u^1,\ldots, u^k$ on $V$ that are homogeneous of degree one. Then we
  have a splitting $V = \mathbb{R}^k \times Y$.
\end{prop}
\begin{proof}
By assumption we have a sequence $B(p_i, 2)$ of balls in Riemannian
manifolds with
$\mathrm{Ric} > -i^{-1}$, such that $B(p_i,2)\to B(o,2)$ in the
Gromov-Hausdorff sense. We will prove the following: for any $\delta >
0$, we can find an $r > 0$ and $\delta$-splitting maps $u_i : B(p_i,
r) \to \mathbb{R}^k$ for sufficiently large $i$. Since $V$ is a cone, after
scaling up by $r^{-1}$ and taking a diagonal sequence, we find a
sequence $B(p_i', 1)\to B(o, 1)$ such that each $B(p_i', 1)$ admits an
$i^{-1}$-splitting map.  From this it follows that $B(o,1/2)$ splits
an isometric factor of $\mathbb{R}^k$. For this, see
Cheeger-Colding~\cite{CC}, or
Cheeger-Naber~\cite[Definition 1.20, Lemma 1.21]{CN3}.

Before we begin let us recall the notion of a $\delta$-splitting map. A map $u=(u^1,
\ldots, u^k):
B(p,r) \to \mathbb{R}^k$ is a $\delta$-splitting map, if it is
harmonic, and satisfies
\begin{enumerate}
\item $|\nabla u| < 1 + \delta$,
\item $\dashint_{B_r(p)} |\, \langle \nabla u^\alpha, \nabla
  u^\beta\rangle - \delta^{\alpha\beta}|^2 < \delta^2$,
\item $r^2 \dashint_{B_r(p)} |\nabla^2 u^\alpha|^2 < \delta^2$.
\end{enumerate}

Consider again our sequence $B(p_i, 2) \to B(o,2)$. We can assume
that
\[ \dashint_{B(o,2)} \langle \nabla u^\alpha,\nabla u^\beta\rangle =
\delta^{\alpha\beta}, \]
and since the $u^\alpha$ are homogeneous, this implies that for all
$r$ we have
\[ \dashint_{B(o,r)} \langle\nabla u^\alpha,\nabla u^\beta \rangle=
\delta^{\alpha\beta}. \]
 We can find a
sequence of harmonic functions $u_i^{\alpha}$ on $B(p_i, 2)$ such that
under the Gromov-Hausdorff convergence we have $u_i^\alpha \to
u^\alpha$ uniformly on each compact set, and moreover for any $0 < r < 2$,
\begin{equation} \label{eq:limL2} \lim_{i\to\infty} \dashint_{B(p_i,r)} \langle \nabla u_i^\alpha, \nabla
u_i^\beta\rangle =
\delta^{\alpha\beta}. \end{equation}
Let $f_i$ denote a harmonic function of the form $u^\alpha_i$ or
$\frac{1}{\sqrt 2} (u_i^\alpha \pm u_i^\beta)$ for $\alpha\not=
\beta$. By the Bochner formula $\Delta |\nabla f_i|^2 \geq -\Psi(i^{-1})
|\nabla f_i|^2$, and so by the mean value inequality, for sufficiently
large $i$ we have
\[ \sup_{B(p_i, 1.5)} |\nabla f_i|^2 \leq C. \] It follows that for any $0 < r < 1.5$ we have
\begin{equation} \label{eq:conv}\lim_{i\to\infty} \dashint_{B(p_i,r)} |\nabla f_i|^2 = 1. \end{equation}
As the gradient of $f_i$ is uniformly bounded, we find the above convergence is uniform on the interval $a<r<1$, where $a>0$ is any constant.

Note that $\sup_{B(p_i, r)} |\nabla f_i|^2 \geq 1/2$, and so
for large $i$
\[ \sup_{B(p_i,1)} |\nabla f_i|^2 \leq 2C \sup_{B(p_i,r)} |\nabla
  f_i|^2. \]
Given $\epsilon > 0$, we can then choose $r_0 > 0$ depending on
$\epsilon, C$, such that for all sufficiently large $i$ there is some
$r \in (r_0, 1/10)$, perhaps depending on $i$, satisfying
\[ \sup_{B(p_i, 3r)} |\nabla f_i|^2 \leq (1-\epsilon)^{-1} \sup_{B(p_i,
    r)} |\nabla f_i|^2. \]
Consider now the functions $v_i = \sup_{B(p_i, 3r)}|\nabla f_i|^2 -
|\nabla f_i|^2$. Then on $B(p_i, 3r)$, $v_i \geq 0$,
\[ \Delta v_i \leq \Psi(i^{-1}) \sup_{B(p_i, 3r)}
|\nabla f_i|^2= \Psi(i^{-1}), \]
 and
\[ \inf_{B(p_i, r)} v_i = \sup_{B(p_i, 3r)}|\nabla f_i|^2 -
  \sup_{B(p_i, r)} |\nabla f_i|^2 \leq \epsilon \sup_{B(p_i, 3r)}
  |\nabla f_i|^2. \]
From the weak Harnack inequality, once $i$ is sufficiently large,
\[ \dashint_{B(p_i, 2r)} v_i \leq C(\Psi(i^{-1})+ \epsilon
  \sup_{B(p_i, 3r)} |\nabla f_i|^2) \leq 2\epsilon C
  \sup_{B(p_i, 3r)} |\nabla f_i|^2. \]
This implies
\[ (1-2C\epsilon) \sup_{B(p_i, 3r)} |\nabla f_i|^2 \leq \dashint_{B(p_i,
    2r)} |\nabla f_i|^2,  \]
where $C$ depends only on the non-collapsing constant, through the
Sobolev inequality. Recall $r$ (depending on $i$) has a lower bound $r_0$, and so by (\ref{eq:conv}),
 \[ \lim_{i\to\infty} \dashint_{B(p_i,2r)} |\nabla f_i|^2 = 1. \]

It follows that for sufficiently large $i$ we have
\[\sup_{B(p_i, 2r_0)}|\nabla f_i|^2\leq \sup_{B(p_i, 3r)}|\nabla f_i|^2 \leq 1+C\epsilon.\]
Therefore,
\[ \dashint_{B(p_i, 2r_0)}||\nabla f_i|^2-1|<\Psi(\epsilon).\]
 We now apply this to $f_i = u^\alpha_i$, and to $f_i = \frac{1}{\sqrt{2}}(u_i^\alpha
\pm u_i^\beta)$, and use the polarization identity
\[ \frac{1}{2} |\nabla (u_i^\alpha + u_i^\beta)|^2 - \frac{1}{2}
  |\nabla (u_i^\alpha - u_i^\beta)|^2 = 2 \langle \nabla u_i^\alpha,
  \nabla u_i^\beta\rangle. \]
Using also \eqref{eq:limL2} we find that for sufficiently large $i$ (depending on
$\epsilon$), we have
\[ \dashint_{B(p_i, 2r_0)} \Big| \langle \nabla u_i^\alpha, \nabla
  u_i^\beta \rangle - \delta^{\alpha\beta} \Big| < \Psi(\epsilon). \]
Since at the same time $|\nabla u_i^\alpha|^2 \leq 1 + \Psi(\epsilon)$, we
can use the Bochner formula, using a cutoff function $\phi$ as in
Cheeger-Colding~\cite{CC} supported in $B(p_i, 2r_0)$, equal to 1 in
$B(p_i, r_0)$. We find that for sufficiently large $i$,
\[ \begin{aligned}  \dashint_{B(p_i, r_0)} |\nabla^2
    u^\alpha_i|^2 &\leq \dashint_{B(p_i, 2r_0)} \frac{1}{2} \phi \Delta (|\nabla
    u^\alpha_i|^2 -1) - \Psi(1/i) \dashint_{B(p_i, 2r_0)} \phi
    |\nabla u^\alpha_i|^2 \\
    &\leq C r_0^{-2} \dashint_{B(p_i, 2r_0)} \Big| |\nabla
    u^\alpha_i|^2 - 1\Big| - \Psi(1/i) \\
    &\leq r_0^{-2} \Psi(\epsilon).
  \end{aligned}\]

If $\epsilon$ is chosen sufficiently small, depending on $\delta > 0$,
then this shows that $u_i = (u_i^1,\ldots, u_i^k)$ is a
$\delta$-splitting map on $B(p_i, r_0)$, for sufficiently large $i$.
\end{proof}

\end{document}